\documentclass[10pt]{amsart}

\usepackage{amsmath}
\usepackage{amssymb}
\usepackage{enumerate}
\usepackage{amsthm}
\usepackage{tikz-cd}
\usepackage{mathrsfs}
\usepackage[normalem]{ulem}
\usepackage{caption}
\usepackage{bigints}
\usepackage{parskip}
\usepackage{esint}
\usepackage{bbm}
\usepackage{dsfont} %Josh added
\usepackage[dvipsnames]{xcolor}
\usepackage{hyperref}
\hypersetup{
    colorlinks=true,
    linkcolor=blue,
    citecolor=cyan,
    filecolor=magenta,      
    urlcolor=cyan,
    pdftitle={Lattice Actions},
    pdfpagemode=FullScreen,
    }
\usepackage{geometry} 
\usepackage{setspace,kantlipsum} %spacing
\usepackage{thmtools} % for restating theorems without changing numbering
\usepackage{thm-restate} % for restating theorems without changing numbering
\usepackage{verbatim} %multiline commenting
\usepackage{tikz}
\usetikzlibrary{matrix,arrows}
\usepackage{graphicx}

\newtheorem{thm}{Theorem}[section]

\newtheorem{coro}[thm]{Corollary}
\newtheorem{lem}[thm]{Lemma}
\newtheorem{prop}[thm]{Proposition}

\newtheorem{quest}[thm]{Question}
\newtheorem{lemma}[thm]{Lemma}

\theoremstyle{definition}

\newtheorem{eg}[thm]{Example}

\theoremstyle{remark}

\newtheorem{remk}[thm]{Remark}

\input{Chris-macros.sty}

\newcommand{\oalpha}{\overline{\alpha}}

\newcommand{\oF}{\overline{F}}
\newcommand{\oK}{\overline{K}}

\newcommand{\Teich}{{\rm Teich}}

\newcommand{\tK}{\widetilde{K}}
\newcommand{\tPhi}{\widetilde{\Phi}}
\newcommand{\tN}{\widetilde{N}}
\newcommand{\tmu}{\widetilde{\mu}}

\newcommand{\Per}{{\rm Per}}
\newcommand{\sys}{{\rm syst}}

\newcommand{\mrev}[1]{\href{https://mathscinet.ams.org/mathscinet-getitem?mr=#1}{MR#1}}
\newcommand{\zbl}[1]{\href{https://zbmath.org/#1}{Zbl #1}}
\newcommand{\doi}[1]{doi:\,\href{http://dx.doi.org/#1}{#1}}

\definecolor{darkred}{rgb}{0.55, 0, 0}

\title[Affine mappings: shrinking targets and Diophantine approximation]{Affine mappings of translation surfaces:\\
shrinking targets and Diophantine approximation}

\author{Chris Judge}
\address{Indiana University, Bloomington, IN 47405}
\email{cjudge@iu.edu}

\author{Josh Southerland}
\address{University of Bristol, Bristol, England BS8 1UG}
\email{josh.southerland@bristol.ac.uk}


\begin{document}

\begin{abstract}
Let $(X,\omega)$ be a translation surface whose Veech group $\Gamma_\omega$ is a lattice. We prove that the generic orbit of the group of affine homeomorphisms of $(X,\omega)$ can be used to approximate each point of $X$ with Diophantine precision. The proof utilizes an induced $SL_2(\Rbb)$-action on a fiber bundle $Y_\omega$ whose base is $SL_2(\Rbb)/\Gamma_\omega$ and whose fiber is a quotient of $X$.
We observe that this bundle embeds as an $SL_2(\Rbb)$-orbit closure in the moduli space of once marked translation surfaces, and hence we may invoke the spectral gap results of Avila and Gou\"ezel \cite{Avila-Gouezel-13-small-eigenvalues} 
and a quantitative mean ergodic theorem 
for the $SL_2(\Rbb)$-action on the mean-zero, square-integrable functions on $Y_\omega$.
\end{abstract}

%\setstretch{1.25}

\maketitle

\section{Introduction}
\label{sec:intro}

Let $X$ be a connected, closed, orientable, surface of genus at least two, and let $\omega$ be a nontrivial
complex valued 1-form that is holomorphic with respect to some complex structure on $X$. Integration of $(i/2) \omega \wedge \overline{\omega}$ defines a finite measure $\mu_\omega$ on $X$. Integration of $|\omega|$ defines a notion of length that in turn defines a distance function $d_\omega$. A self-homeomorphism $\phi:X \to X$ 
is said to be {\it affine} with respect to $\omega$ iff 
\begin{equation}
\label{eqn:affine-re-im}
\begin{pmatrix}
\Re(\phi^*(\omega)) \\ 
\Im(\phi^*(\omega)) 
\end{pmatrix}
~
=~
\begin{pmatrix}
a & b \\ c & d
\end{pmatrix}
\begin{pmatrix}
\Re(\omega) \\ 
\Im(\omega) 
\end{pmatrix}
\end{equation}
for some $2 \times 2$ invertible real matrix. The matrix is uniquely determined 
by (\ref{eqn:affine-re-im}), and we will denote it by $D \phi$. For each
affine homeomorphism $\phi$ we have $\mu_\omega(X)= \mu_\omega(\phi(X))= |\det(D\phi)|\,\mu_\omega(X) < \infty$,
and hence $\det(D\phi)=\pm 1$. It follows that the map $\phi \mapsto D\phi$ defines a homomorphism from the group, $\Aff^+_\omega$, of orientation-preserving 
affine homeomorphisms into $SL_2(\Rbb)$. 

 In the present paper, we assume  that the image of $D$ is a lattice 
 subgroup of $SL_2(\Rbb)$. In this case, most elements in $\Aff^+_{\omega}$ are pseudo-Anosov diffeomorphisms and hence act ergodically---in fact, Bernoulli---on $(X, \mu_\omega)$ (see \S 10.6 in \cite{FLP79}). Therefore, the action of $\Aff^+_{\omega}$ is ergodic, and so we can ask questions about the quantitative density of the orbits of generic points.

For each $g \in SL_2(\Rbb)$, let 
$\|g\|:= \sqrt{\tr(g^tg)}$. To measure the quantitative density of orbits, 
for each $(x,y, \alpha) \in X \times X \times \Rbb$ we consider the
set
$$
H_{x,y,\alpha}~ 
:=~
\left\{
\phi \in \Aff_{\omega}^+~
:~ 
d_{\omega}(\phi (x), y)~
\leq~
\| D \phi \|^{-\alpha} \right\}.
$$
One can show that there exists a constant $\alpha_\omega$ so that 
for almost every $(x,y) \in X \times X$ the set $H_{x,y,\alpha}$
is infinite if $\alpha < \alpha_\omega$ and finite if $\alpha > \alpha_\omega$ 
(compare with \cite{Laurent-Nogueira-12} \cite{Kelmer-17} and see Proposition
\ref{prop:alpha-omega}).

\begin{thm} 
\label{thm:full-measure}
Suppose that $D\Aff^+_\omega$ is a lattice in $SL_2(\Rbb)$. Then 
$0 < \alpha_\omega \leq 1$.
\end{thm}

If  $X$ were the torus $\Tbb^2:=\Cbb/(\Zbb + i \Zbb)$ marked at $0$ and $\omega= dz$,
then it would follow from work of Ghosh, Gorodnik, and Nevo that $\alpha_\omega=1$ 
 (see Corollary 1.2 in \cite{GGN}). In the language of \cite{GGN}, the constant $\alpha_\omega$ is the `Diophantine approximation exponent', and if 
 $\alpha_\omega =1$, then one says
 that this exponent is optimal \cite{GGN}. 

 Using a different method, Finkelshtein \cite{Fink17} 
 recovered the above result of \cite{GGN} and extended it to the 
 action of thin subgroups of $\Aff^{+}_{dz} \cong SL_2(\Zbb)$ acting on the torus
 with a refinement of the quantitative estimates. 
 The second named author extended these estimates to `regular' `square-tiled surfaces'
 \cite{Sou24}. Recall that $(X, \omega)$ is a square-tiled surface if and only if there exists 
 a branched covering $p: X \to \Tbb^2$ branched over $0$ so that $\omega=p^*(dz)$. A square-tiled surface is said to be `regular' if $p|_{X \setminus p^{-1}(0)}$ is a normal
 covering. The results in \cite{Sou24} imply that $\alpha_\omega =1$ if $(X,\omega)$ is
 a regular square-tiled surface.

The proof of the upper bound in Theorem \ref{thm:full-measure} 
comes from an elementary argument involving the `easy' implication of the Borel-Cantelli Lemma. Our proof of the strict lower bound on $\alpha_\omega$ depends on showing that
a certain unitary representation of $G:=SL_2(\mathbb{R})$ has a spectral gap.
(Compare with, for example, \cite{GGN} \cite{Kelmer-17}.) 

This unitary representation of $G$ corresponds to the action of $G$ on the space $Y_\omega$ obtained by quotienting $G \times X$ by the natural right action of $\Aff^+_\omega$.
To be precise, $\Aff^+_\omega$ acts on the right of $G$ by $g \cdot \phi =
g\, D\phi$ and on the right of $\dot{X}_\omega:=X \setminus \omega^{-1}(0)$ 
by $x\cdot \phi = \phi^{-1}(x)$. 
The product of Haar measure $dm$ and $\mu_\omega$ is invariant under the
diagonal action of $\Aff^+_\omega$ and so the product measure descends to a 
measure $\nu$ on $Y_\omega:= (G \times \dot{X}_\omega)/ \Aff^+_\omega$. Moreover, the left 
action of $G$ on $G \times \dot{X}_{\omega}$ given by $h\cdot (g,x)= (hg, x)$ determines
a left action of $G$ on $Y_\omega$ that preserves $\nu$ (see \S \ref{sec:aff-manifold}
for details). 
Therefore we have 
a unitary representation $\pi_\omega$ of $G=SL_2(\Rbb)$ on $L^2(Y_\omega,\nu)$ defined
by 
\begin{equation}
\label{eqn:pi-omega}
(\pi_\omega(g) u)([h,x])~ :=~ u\left(g^{-1} \cdot [h,x] \right)
\end{equation} where $L^2(Y_\omega,\nu)$ is the space of $u$ that are square-integrable with respect to $\nu$.

The isomorphism classes of 
irreducible unitary representations of $SL_2(\Rbb)$ are well-known
\cite{GGP}. In  \S \ref{sec:AG} we give a description of the spherical 
irreducible representations $\pi_s$. In this description,  $\pi_1$ is the trivial representation, $\pi_s$ is a complementary series iff $0<s<1$, and $\pi_s$ with $\Re(s)=0$ is the spherical principal series.
Define $\beta_\omega$ to be the infimum of $1-s$ with $0\leq s < 1$ such that
 $\pi_s$ occurs in the decomposition of $\pi_\omega$.
 The constant $\beta_\omega$ is the {\it spectral gap} for $\pi_\omega$.\footnote{
 The lattice $\Gamma_\omega$ is not uniform \cite{Veech89}.
 Hence each $\pi_s$ with $\Re(s)=0$  
appears in the decomposition of $L^2_0(G/\Gamma_\omega)$, and in particular $\pi_0$ appears.  Since 
$L^2_0(G/\Gamma_\omega)$ pulls back to $L^2_0(Y_\omega,\nu)$, it follows that 
$\beta_\omega \leq 1$.
}

The lower bound in Theorem \ref{thm:full-measure} will follow from

\begin{thm}
\label{thm:alpha-beta}
We have $\alpha_\omega \geq \beta_\omega  > 0$. 
\end{thm}

In fact, we prove a slightly stronger statement. 
We show that for every $y$ and a.e. $x$ the set $H_{x,y,\alpha}$ is 
infinite if $\alpha < \beta_\omega$. 
See Theorem \ref{thm:diophantine}.

\begin{coro}
If $\pi_\omega$ does not contain a complementary series representation, then the Diophantine approximation exponent is optimal.
\end{coro}

 \begin{quest}
For which $(X,\omega)$ does $\pi_\omega$ not contain a complementary series representation?
\end{quest}

In general, the spectral gap $\beta_\omega$ does not equal 1. Indeed,
the space $Y_\omega=(G \times \dot{X}_\omega)/\Aff_{\omega}^+$ is a bundle over $G/D\Aff_{\omega}^+,$ and the left action of $G$ commutes with the bundle projection.
Therefore, any irreducible complementary series representation of $L^2(G/D\Aff_{\omega}^+)$ would pull back to an irreducible complementary series representation in the bundle $L^2(Y_\omega)$. Ellenberg and McReynolds show that for any finite index subgroup $\Gamma \subset \Gamma(2)$ of the principal congruence group of level 2 that contains $-\mathrm{Id}$, there exists a square-tiled surface $(X, \omega)$ 
such that $D\Aff_{\omega}^+ = \Gamma$~\cite{ElMc12}~\cite{J13}.
The group $\Gamma(2)/\{\pm \rm Id \}$ is isomorphic to a free group, and so 
by abelianization and reduction modulo $n$, one obtains a surjective homomorphism 
$\Gamma(2) \to \Zbb/n\Zbb$. For large enough $n$ the kernel has complementary 
series (see Appendix A of \cite{MWS13}). 
Moreover, Matheus and Weitze-Schmith\"usen have constructed explicit examples of $(X,\omega)$ such that $L^2(G/D\Aff_{\omega}^+)$ has irreducible complementary series representations with $s$ arbitrarily close to 1 yielding a $\beta_{\omega}$ arbitrarily close to zero \cite{MWS13}.

\begin{quest}
Does there exist $(X, \omega)$ so that $\alpha_\omega \neq \beta_\omega$? 
\end{quest}

Since $\alpha_\omega =1$ for regular square-tiled surfaces \cite{Sou24}, the answer 
is `yes' if one could construct a regular square-tiled surface $(X,\omega)$ so that
the natural representation of $G$ on $L^2(G/D \Aff_{\omega}^+)$ 
has a complementary series.

To prove Theorem \ref{thm:alpha-beta}, we 
adapt the method of Ghosh, Gorodnik, and Nevo \cite{GGN} (see also \cite{Kelmer-17}). We frame the Diophantine question in terms of a shrinking target problem and move this problem into the context of the bundle $Y_\omega = (G \times \dot{X}_\omega)/ \Aff^+_{\omega}$. 
We show that a solution to a shrinking target problem with respect to the $G$-action on $Y_\omega$
gives a solution to the shrinking target problem for the affine group action on the surface $X$. A spectral gap 
allows us to use a quantitative mean ergodic theorem (see Appendix \ref{sec:MET})
in combination with a Borel-Cantelli argument to prove that $\alpha_\omega \geq \beta_\omega$.   

To show that $\beta_\omega > 0$, we embed
the total space of the bundle $Y_\omega$ as a $G$-orbit closure $\Mcal$ in the moduli space of marked translation surfaces. 
The embedding intertwines the measure-preserving $G$-action on $Y_\omega$ and the measure-preserving $G$-action on $\Mcal$. Therefore, we may apply the results concerning spectral gap in \cite{Avila-Gouezel-13-small-eigenvalues} to obtain $\beta_\omega >0$.

%%%%%%%%%%%%%%%%%%%%%%%%%%%%%%%

\subsection{The affine group}

The study of the affine group $\Aff^+_\omega$  has its origins in the work of Thurston  and Veech. Thurston \cite{Th88} constructed 
examples of affine groups generated by two multitwists. Using regular 
polygons Veech \cite{Veech89} constructed examples of translation surfaces 
so that $D\Aff^+_{\omega}$ is a lattice. Since that time many authors have discovered
many new examples of translation surfaces with this lattice property. We refer  
the reader to the recent survey of McMullen \cite{McMullen}.

If $P \subset \Cbb$ is a planar polygon whose vertex angles are rational multiples of $\pi$, then there exists a surface $(X, \omega)$ such that $(P,dz)$ is obtained by quotienting $(X, \omega)$ by a finite group of isometries  \cite{FK36} \cite{ZK75}.
If $D \Aff^+_\omega$ is a lattice, then the dynamics of $\Aff_{\omega}^+$ is intimately connected with the billiard flow on $P$. Indeed, Veech realized that if $D \Aff^+_\omega$ is a lattice, then $\Aff^+_\omega$ acts with finitely many orbits  on the set of closed trajectories (up to homotopy) of the  billiard flow lifted from $P$ \cite{Veech89}. He used this observation to obtain asymptotics on the lengths of closed billiard trajectories.

\subsection{Lattices and Homogeneous dynamics}
Lattice actions on homogeneous spaces have been extensively studied in part due to their close relationship with number theory and various counting problems.
The particular case of approximating points in the punctured plane $\Rbb^2\setminus \{0\}$ with orbits of $SL_2(\mathbb{Z})$ acting linearly on the punctured plane was studied by Ledrappier using dynamical methods~\cite{Le99}, and by Nogueira using number-theoretic methods~\cite{Nogueira-02}. 
Later, Laurent and Nogueira used number-theoretic methods to improve our understanding of approximation at particular points achieving the best possible rates for approximating rational vectors~\cite{Laurent-Nogueira-12}. Using dynamical methods, Maucourant and Weiss studied the action of cocompact lattices~\cite{MW12}. They also approximate rates for $SL_2(\mathbb{Z})$-orbits of vectors possessing a particular Diophantine condition, but these rates are not as strong as those produced in~\cite{Laurent-Nogueira-12}. Later still, Kelmer applied the framework of Ghosh, Gorodnik, and Nevo to general lattices $\Gamma \subset SL_2(\mathbb{R})$ acting linearly on the plane~\cite{Kelmer-17}. His work yields the best possible rates whenever the lattices in question are tempered.\footnote{Note that a lattice $\Gamma$ in 
$G:=SL_2(\Rbb)$ is {\it tempered} iff the natural unitary representation of $G$ 
on $L^2(G/\Gamma) \ominus \Cbb$
contains no complementary series.}

The dynamical methods used in this literature share a few common features. First, the use of a closed subgroup $H$. Ledrappier \cite{Le99} and Maucourant and Weiss \cite{MW12} take $H$ to be a unipotent subgroup. 
For any lattice $\Gamma \subset SL_2(\mathbb{R}),$ the action of this one-parameter subgroup $H$ on $SL_2(\mathbb{R})/\Gamma$ is dual to the action of the lattice $\Gamma$ on $H\setminus SL_2(\mathbb{R}) \cong \mathbb{R}^2\setminus \{0\}.$ The work is to translate quantitative information for the action of $H$ into quantitative information for the lattice action. Ghosh--Gorodnik--Nevo consider this duality with respect to any closed subgroup $H$ of an algebraic group $G$. Then a lattice $\Gamma \subset G$ acting on $H\setminus G$ is dual to the action of $H$ on $G/\Gamma.$ Their results apply to a variety of settings beyond approximation in the punctured plane and yield optimal Diophantine exponents when the lattice $\Gamma$ is tempered.

Second, at the core of the arguments is a ``shrinking target" framework. For the benefit of the reader, we describe a shrinking target problem in our setting. Fix a lattice surface $(X,\omega)$ with orientation-preserving affine homeomorphisms $\Aff_{\omega}^+$ and pick any $y \in X$. For any $\phi \in \Aff_{\omega}^+,$ let $B_{r(\phi)}(y)$ denote the open ball with a radius $r$ that depends on the operator norm of the differential of $\phi.$ These are our targets for each respective $\phi.$ We can ask: does almost every $x \in X$ have the property that $\phi (x) \in B_{r(\phi)}(y)$ for infinitely many $\phi \in \mathrm{Aff}_{\omega}^+$? How fast can we decrease the radius (shrink the target) as the norm of the differential of $\phi$ increases before this no longer holds? 
Hill and Velani coined the term ``shrinking target'' in their fundamental work on the subject~\cite{HiVe95}. The first appearance of the framework is generally agreed to be in the work of Philipp~\cite{Phi67}. The shrinking target framework is now ubiquitous in literature exhibiting dynamical methods to deduce Diophantine properties or logarithm laws.

\subsection{Organization} 
In Section \ref{sec:duality}, we review techniques for establishing Diophantine properties for lattice actions. In Section \ref{sec:1forms-and-affine-diffeos}, we define and review properties of affine diffeomorphisms and the closely associated Veech group. In Section \ref{sec:AG}, we provide background concerning strata of holomorphic 1-forms, discuss the spherical representations of 
$SL_2(\Rbb)$ and then describe the main result of \cite{Avila-Gouezel-13-small-eigenvalues} in a form 
that is applicable to our setting. In Section \ref{sec:aff-manifold} we discuss the bundle $Y_\omega$, embed $Y_\omega$ as an affine invariant manifold in a stratum, and apply the result of \cite{Avila-Gouezel-13-small-eigenvalues} to show that 
$\pi_\omega$ has a spectral gap. In Section \ref{sec:shrinking}, we translate the spectral gap into an effective mean ergodic theorem for the $SL_2(\mathbb{R})$-action using a standard theorem that we cite in Appendix \ref{sec:MET}. We use the results of Section \ref{sec:duality} alongside further Borel-Cantelli arguments to establish the Diophantine properties cited in Theorem \ref{thm:diophantine}. 
In Appendix \ref{app:avila-gouezel}, we explain why the result of Avila and Gou\"ezel applies to strata with marked points.

\subsection{Funding} The work of C. J. was supported in part by the  Leverhulme Trust, the Simons Foundation, and a research professorship at Le Studium. The work of J. S. was supported by the Heilbronn Institute for Mathematical Research.

\subsection{Acknowledgements}\label{acknowledgements}
Thanks to Paul Apisa, Jon Chaika, Alex Furman, Vaibhav Gadre, Chris Johnson, Dami Lee, and Carlos Matheus for helpful discussions. Thanks to Alex Gorodnik for providing the reference \cite{Ananetal}. The second author thanks Jayadev Athreya for introducing him to shrinking target problems.  We thank Curt McMullen for his helpful 
comments concerning the introduction. Finally, we are in the referee's debt for corrections 
and suggestions.

%%%%%%%%%%%%%%%%%%%%%%%%%%%%%%%%%%%%%%%%%%%%%%

\ \\

\section{Diophantine properties for lattice actions}
\label{sec:duality}

\medskip

Let $G$ be a locally compact, Hausdorff group.  
Let $m$ denote a Haar measure on $G$.
Let $\Gamma \subset G$ be a lattice, that is, a discrete subgroup so that $m(G/\Gamma) < \infty$.
Let $X$ be a space equipped with a finite measure $\mu$.
We suppose that $\Gamma$ acts on the left of $X$ so that $\gamma_*(\mu)=\mu$ for each $\gamma \in \Gamma$.

Let us describe the shrinking target problem in this context.
Roughly speaking, we are given a `target' $T \subset X$ with $\mu(T) >0$.
If $\Gamma$ acts $\mu$-ergodically on $X$, then for a.e. $x$ there exists $\gamma$ such that $\gamma \cdot x \in T$. 
The shrinking target problem asks: How does the `size' of $\gamma$ depend on $\mu(T)$ as $\mu(T) \searrow 0$. 

Here we will measure the size of $\gamma$ using a function $\rho$ that is compatible with the group structure.
Namely, we suppose that there exists a continuous 
function $\rho:G \to [0,\infty)$ and a constant $C_\rho$ so that
$\rho(g \cdot h) \leq \rho(g) + \rho(h) +C_\rho$ 
and $\rho(g^{-1}) \leq \rho(g) + C_\rho$. In \S \ref{sec:shrinking} where we use the results of this section, we will take $\rho$ to be the log of the Frobenius norm with $C_{\rho} = 0.$

We study the $\Gamma$-action on $X$ by studying a $G$-action on a `fiber bundle' $Y$ over $G/\Gamma$ with fiber $X$.  
We define the space $Y$ to be the quotient $(G \times X)/\Gamma$ where $\Gamma$ acts on the right of $G \times X$ by 
$$
(g,x)\cdot \gamma~
:=~
(g \cdot \gamma, \gamma^{-1}\cdot x).
$$
The group $G$ acts on the left of $Y$ by 
$$
g (h,x) \Gamma~ 
:=~ 
(g\cdot h, x) \Gamma.
$$
Let $\Fcal$ be a measurable fundamental domain for the 
$\Gamma$-action on $G \times X$. 
Define $\nu$ to be the measure $\pi_*(m \times \mu)$ where $\pi: \Fcal \to Y$ is the quotient map.
The $G$-action on $Y$ preserves the measure $\nu$.

To understand the shrinking target problem for the action of 
$\Gamma$ on $X$, we consider a related shrinking target problem for the 
action of $G$ on $Y$. In particular, we choose an auxiliary compact 
set $S \subset G$ with $m(S)>0$, and we ask: ``In order for $g \cdot y$ to lie in $\pi(S \times T)$, how `large' should $g$ be?''
The following proposition shows that an estimate for $G$ acting on $Y$ 
determines an estimate for $\Gamma$ acting on $X$.

Define  $G_t:= \rho^{-1}[0,t]$ for $t \in [0, \infty)$ and define $\Gamma_t := G_t \cap \Gamma$.

\begin{prop}
\label{prop:good-project-good}
\ 
Let $y= (h, x) \Gamma \in Y$.  
If $G_t^{-1} \cdot y\,  \cap\, \pi(S \times T) \neq \emptyset$, then 
$\Gamma_{t+C_h}\cdot  x \cap T \neq \emptyset$ where
$C_h = \rho(h) + \sup_{k \in S} \rho(k) + 4C_\rho$.
\end{prop}

\begin{proof}
If $G_t^{-1} \cdot (h,x) \Gamma \cap \pi(S \times T) \neq \emptyset$, then there exists $g \in G_t$ and $\gamma \in \Gamma$ 
such that 
$$
(g^{-1}h,x) \gamma~
\in~
S \times T.
$$
Thus, $g^{-1}h \gamma \in S$ and $\gamma^{-1} x \in T$.
In particular, $\gamma^{-1} = s^{-1} g^{-1}h$ for some $s \in S$. 
Thus using the properties of $\rho$ we find that 
\begin{eqnarray*}
\rho(\gamma^{-1})~
%&\leq&
%\rho(\gamma^{-1})\, +\, C_\rho
%\\
&\leq&
\rho(s^{-1})\, +\, \rho(g^{-1})\, +\, \rho(h)\, +\, 2C_\rho
\\
&\leq&
\rho(s)\, +\, \rho(g)\, +\, \rho(h)\, +\, 4C_\rho
\\
&\leq&
t\, +\, \rho(h)\, +\, \sup \{\rho(s)\, :\, s \in S\}\, +\, 4C_\rho.
\end{eqnarray*}
The claim follows.
\end{proof}

\medskip

Let $\Bcal_{S,T,t}$ denote the set of 
$y\in Y$ 
such that $G_t^{-1} \cdot y$ does not intersect $\pi(S \times T)$.

\begin{lem}[Borel-Cantelli argument]
\label{lem:bc-duality}
Let $t_n$ be a sequence of positive numbers.
Suppose that there exists a sequence of sets $T_n$ so that 
$
   \sum_n \nu( \Bcal_{S,T_n,t_n})
  ~
  <
  \infty.
$
For a.e. $x \in X$ there exists $n_x$ so that if $n\geq n_x$, then 
there exists $\gamma \in \Gamma$ so that 
\begin{enumerate}
\item $\gamma \cdot x \in T_{n}$.

\item  $\rho(\gamma)\, \leq\,  t_n\, +\, 2 \sup _{s \in S} \rho(s)\, +\, 4C_\rho$
\end{enumerate}
\end{lem}

\begin{proof}
Borel-Cantelli implies that $\nu(E)=0$ where
$$
E~
=~
\bigcap_{j=1}^{\infty}
\bigcup_{n=j}^{\infty}
\Bcal_{S,T_n,t_n}.
$$
Let $\Fcal'$ be a fundamental domain for the action of $\Gamma$ on $G$ so that $m(\Fcal' \cap S) >0$.
Let $\tE$ be the subset of $\Fcal' \times X$ 
so that $(h,x) \mapsto (h,x) \Gamma$ maps $\tE$ onto $E$.
Then $(m \times \mu)(\tE)= \nu(E)=0$.
Let $F$ be the set of $x\in X$ such that $m(\{h  \in \Fcal'\,:\, (h,x) \in \tE \})>0$.
Fubini's theorem implies that $\mu(F)=0$.  

Suppose that $x \notin F$, and so $m(\{h  \in \Fcal'\,:\, (h,x) \in \tE \})=0$.
Since $m(\Fcal' \cap S)>0$, there exists $h \in \Fcal' \cap S$ so that $(h,x) \notin \tE$.
In other words, 
$$
(h,x) \Gamma~
\in~
Y \setminus E~
=~
\bigcup_{j=1}^{\infty} \bigcap_{n=j}^{\infty} Y \setminus \Bcal_{S,T_n,t_n}. 
$$
Hence there exists $n_x$ such that 
$$
(h,x) \Gamma~ 
\in~
\bigcap_{n=n_x}^{\infty} Y \setminus \Bcal_{S,T_n,t_n}.
$$
Therefore, $y =(h,x) \Gamma \in  Y \setminus \Bcal_{S,T_n,t_n}$
for each $n \geq n_x$.
In other words, for $n \geq n_x$, we have
$G_{t_n}^{-1} \cdot y\, \cap\,  \pi(S \times T_n) \neq \emptyset$.

Proposition \ref{prop:good-project-good} then implies that 
for $n \geq n_x$ we have $\Gamma_{t_n+C} \cdot x\, \cap\, T_n \neq \emptyset$
where $C= 2\sup_{s \in S} \rho(s)+ 4C_\rho$. The claim follows.
\end{proof}

Recall that $G_t = \rho^{-1}[0,t]$.
Let $u: Y \to \Cbb$ be measurable.
Define the $G_t$ ergodic average of $u$ to be
$$ 
(A_t u)(y)~
:=~
\frac{1}{m(G_t)} \int_{G_t}
u(g^{-1}\cdot y)\, dm(g).
$$
Define the space average of $u$ to be 
$$
\overline{u}~
:=~
\frac{1}{\nu(Y)}\int_{Y} u\, d \nu.
$$
Define the {\em $L^2$ ergodic error function} by
\begin{equation}
\label{eqn:mean-ergodic}
\epsilon(t)~
:=~
\sup_{\substack{u \in L^2(X,\mu)\\ u\not \equiv 0}}
\frac{
\left\|
A_t u- \ou 
\right\|}{\|u\|}
\end{equation}
where $\| \cdot \|$ is the $L^2(Y, d \nu)$ norm.

\begin{lem}
\label{lem:measure-of-bad-1} 
\ 
Suppose that $S \subset G$ is compact, $m(S)>0$, and
$S\cdot \gamma \cap S =\emptyset$
for each 
$\gamma \in \Gamma \setminus \{ {\rm id} \}$.
Then 
$$
\nu( \Bcal_{S,T,t})~
\leq~
\epsilon(t)^2 \cdot \frac{\nu(Y)^2}{m(S)\, \mu(T)}.
$$
\end{lem}

\begin{proof}
Because the $\Gamma$ translates of $S$ are (pairwise) disjoint, the $\Gamma$ translates of $S \times T$ are disjoint.
Hence there exists a measurable set $\Fcal \subset G \times X$ containing $S \times T$ so that $\pi|_{\Fcal}: \Fcal \to Y$ 
is a measurable isomorphism.
If we define $F: Y \to \Rbb$ by $F:=\one_{S \times T} \circ (\pi|_{\Fcal})^{-1}$, then 
\begin{equation}
\label{eqn:period-integral}
\int_{Y} F\, d\nu~
 =~
 m(S) \cdot \mu(T).
\end{equation}

Consider the ergodic average $A_t F$. 
Note that the support of $A_t F$ is disjoint from $\Bcal_{S,T,t}$. 
Indeed, suppose that $y=(h,x) \Gamma \in \Bcal_{S,T,t}$.
Then, by definition, $g^{-1}(h,x) \Gamma$ is not in $\pi(S \times T)$ for each $g \in G_t$.
Hence $F(g^{-1}(h,x) \Gamma)=0$ for each $g \in G_t$, and so $(A_t F)(y)=0$.
\\

Therefore, using  (\ref{eqn:period-integral}), we find that
\begin{eqnarray}
\label{eqn:upper}
\left\|
A_t F
-\,
\oF
\right\|^2
&\geq&
\int_{\Bcal_{S,T,t}}
\left|
A_t F(y)\,
-\, 
\oF
\right|^2 
d \nu(y)
\nonumber
\\
&=&
\left|
\frac{1}{\nu(Y)} \int_{Y} F\, d\nu
\right|^2 
\int_{\Bcal_{S,T,t}} 1\, d\nu
\\
&=&
\frac{1}{\nu(Y)^2}
(m(S)
\cdot
\mu(T))^2
\cdot
\nu(\Bcal_{S,T,t}).
\nonumber
\end{eqnarray}

Note that $F^2=F$, and hence, using 
(\ref{eqn:mean-ergodic}) and (\ref{eqn:period-integral}), we find that
\begin{eqnarray*}
\label{eqn:me-applied}
\left\| 
A_t F~
-~
\oF 
\right\|^2
&\leq&
\epsilon(t)^2
\cdot 
\int_{Y} F^2(y)~ d \nu(y). 
\\
&=&
\epsilon(t)^2 
\cdot 
m(S) \cdot \mu(T).
\nonumber
\end{eqnarray*}

By combining this with (\ref{eqn:upper}), we obtain the claim.
\end{proof}

\begin{prop}
\label{prop:abstract-target}
Let $t_n$ be a sequence of positive numbers.
Let $T_n \subset X$ be a sequence of measurable subsets so that
$
\sum_n \epsilon(t_n)^2 / \mu(T_n)
$
is finite.
Then for almost every $x \in X$ there exists $n_x$ 
such that if $n\geq n_x$, then 
there exists $\gamma \in \Gamma$ so that 
\begin{enumerate}

\item  $\gamma \cdot x \in  T_n$, and 

\item   $\rho(\gamma)~  \leq~   t_n\, +\, 2 \sup_{s \in S} \rho(s)\, +\, 4C_\rho$.
\end{enumerate}
Here  $S \subset G$ is a compact set of positive Haar measure such that
$S\cdot \gamma \cap S =\emptyset$
for each 
$\gamma \in \Gamma \setminus \{ {\rm id} \}$.
\end{prop}

\begin{proof}
Combine Lemma \ref{lem:bc-duality} and Lemma  \ref{lem:measure-of-bad-1}.
\end{proof}

\begin{remk} One can give similar estimates for ``shells'' in $\Gamma.$ Fix a constant $C>0$ and for any $t > C$ define $G_{t,C} = G_t\setminus G_{t-C}.$ Define the shells to be 
$\Gamma_{t,C} := \Gamma \cap G_{t,C}.$ A shell version of Proposition \ref{prop:good-project-good}, in which we deduce that $G_{t,C}^{-1} \cdot y \cap \pi(S \times T) \neq \emptyset$ implies $\Gamma_{t+C_{h}, C+ 2C_{h}}\cdot x \cap T\neq \emptyset$, follows by deducing the additional lower bound $\rho(\gamma^{-1}) \geq t - C - C_h$. Therefore, a shell version of Proposition \ref{prop:abstract-target} which includes the lower bound $\rho(\gamma) \geq t_n - C - 2\sup_{s \in S} \rho(s) - 4 C_\rho$ follows mutatis mutandis. 
\end{remk}

We end this section by providing the argument used to define the constant $\alpha_\omega$
used in Theorem \ref{thm:full-measure}. Though the argument appears to be well-known---see, for example,  \cite{Kelmer-17}---we provide it for the uninitiated reader.

Let $\Gamma$ be an infinite subgroup of the group of measure preserving transformations
of $(X, \mu)$. Suppose that there exists a function 
$\rho:\Gamma \to [0, \infty)$ and a constant $C_\rho$ so that 
$\rho(\gamma \cdot \gamma')\leq  \rho(\gamma) + \rho(\gamma') + C_\rho$.
Let $d: X \times X \to \Rbb$ be a distance function on $X$. 
Define 
$$
H_{x,y,\alpha}~ 
:=~
\left\{
\gamma \in \Gamma~
:~ 
d(\gamma (x), y)~
\leq~
e^{-\alpha \rho(\gamma)} \right\}.
$$ 
Define $\oalpha: X \times X \to \Rbb \cup \{\infty\}$ by setting $\oalpha(x,y)$ equal to 
the supremum of $\alpha$  so that the set $H_{x,y, \alpha}$ has infinite cardinality.

\begin{prop}
\label{prop:alpha-omega} 
Suppose that each sublevel set of $\rho$ is finite. 
Suppose that $\Gamma$ acts ergodically on $(X,\mu)$ and that each $\gamma \in \Gamma$ is Lipschitz with respect to $d$. Then there exists a constant 
$\alpha^*$ so that for $\mu \otimes \mu$ almost every $(x,y)$ we have
$\oalpha(x,y)=\alpha^*$.
\end{prop}

\begin{proof}
Let $\Gamma \times \Gamma$ act on $(X \times X, \mu \otimes \mu)$ by $(\gamma_1, \gamma_2) \cdot (x_1, x_2) := (\gamma_1 \cdot x_1, \gamma_2 \cdot x_2).$ Since $\Gamma$ acts ergodically on $(X,\mu),$ $\Gamma \times \Gamma$ acts ergodically on $(X \times X, \mu \otimes \mu).$ Observe that if $\overline{\alpha}(x,y)$ is $\Gamma \times \Gamma$-invariant, then it is constant almost everywhere. Hence it suffices to show that $\overline{\alpha}(x,y)$ is invariant under the $\Gamma \times \Gamma$-action on $X \times X$.

Fix $\alpha > 0$ and assume $H_{x,y,\alpha}$ is infinite. Define a set 
$$
H_{x,y,\alpha}(\eta) := \left\{ \gamma \in \Gamma : d(\gamma(x), y) \leq e^{-\alpha \rho(\gamma) + \eta} \right\}.
$$ 
We claim that for any $\gamma_1, \gamma_2 \in \Gamma,$ there 
exists an $\eta$ such that the set $H_{\gamma_1 x,\gamma_2 y,\alpha}(\eta)$ is also infinite. 
Indeed, suppose that $\gamma \in H_{x,y,\alpha}$ and set $\delta_{\gamma} = \gamma_2 \gamma \gamma_1^{-1}.$ 
By the hypothesis on $\rho$ we find that $\rho(\gamma) \geq \rho(\delta_{\gamma}) - C$ 
where $C = \rho(\gamma_2) + \rho(\gamma_1^{-1}) + 2C_{\rho}.$ 
Thus, if $C'$ is the Lipschitz constant for $\gamma_2,$ we have $$d(\delta_{\gamma}\gamma_1 x, \gamma_2 y) = d(\gamma_2 \gamma x, \gamma_2 y) \leq C' d(\gamma x, y) \leq C' e^{-\alpha \rho(\gamma)} \leq e^{-\alpha \rho(\delta_{\gamma}) + \alpha C + \ln(C')}.$$ Then $\delta_{\gamma} \in H_{\gamma_1 x, \gamma_2 y,\alpha}(\alpha C + \ln(C')).$ The claim follows with $\eta = \alpha C + \ln(C').$

Now observe that $H_{x,y,\beta}$ is infinite for any $\beta < \alpha$ provided that there exists an $\eta \in \mathbb{R}_{>0}$ such that $H_{x,y,\alpha}(\eta)$ is infinite. Indeed, suppose there exists an infinite sequence $\gamma_n$ that satisfies $d(\gamma_n x,y) < e^{-\alpha \rho(\gamma_n) + \eta}.$ Since each sublevel set of $\rho$ is finite, we must have that $\rho(\gamma_n) \to \infty.$ Thus, if $\beta < \alpha,$ for all sufficiently large $n$ we have that $-\alpha \rho(\gamma_n) + \eta < -\beta\rho(\gamma_n).$ Thus, $H_{x,y,\beta}$ is infinite.

Combining the two observations, it follows that for any $\alpha > 0,$ if $H_{x,y,\alpha}$ is infinite, then for each $\gamma_1, \gamma_2 \in \Gamma,$ the set $H_{\gamma_1 x, \gamma_2 y, \beta}$ is infinite for each $\beta < \alpha.$ Thus, $\overline{\alpha}(\gamma_1 x, \gamma_2 y) \geq \overline{\alpha}(x,y).$ A symmetric argument proves the reverse inequality. Hence $\overline{\alpha}(x,y)$ is $\Gamma \times \Gamma$-invariant, as desired.
\end{proof}

%%%%%%%%%%%%%%%%%%%%%%%%%%%%%%%%%%%%%%%%%%%%%%%%%%%%%%%%%%%%%%%%
%%%%%%%%%%%%%%%%%%%%%%%%%%%%%%%%%%%%%%%%%%%%%%%%%%%%%%%%%%%%%%%%

\newpage 

\ \medskip

\section{Holomorphic 1-forms and affine homeomorphisms}
\label{sec:1forms-and-affine-diffeos} 

\ \medskip

In the remainder of the paper,  $X$ will be
a connected, compact, oriented, topological 
surface without boundary whose
genus is at least two.
A complex structure $J$ on $X$ is an equivalence class of atlases 
$\{\chi_a : U_a \to \Cbb\}$ such that each transition function 
$\chi_a \circ \chi_{b}^{-1}$ is a biholomorphism. We will assume that 
$J$ is compatible with the fixed orientation of $X$.

The complex structure $J$ determines a smooth structure and a tangent bundle.
A $J$-holomorphic 1-form is a holomorphic section of the $J$-holomorphic cotangent bundle. 
Equivalently, a complex-valued 1-form $\omega$ is said to be holomorphic with respect to $J$ 
if and only if for each $\chi_a$, there exists a holomorphic function $f_a: \chi_a(U_a) \to \Cbb$ 
so that $\omega = \left( f_a \circ \chi_a \right)\, \chi_a^*(dz)$ on $U_{a}$
where $dz$ is the standard differential on $\Cbb$. The function $f_a$ is unique. 

If  $x \in U_a$, then we define the
{\it order} of $\omega$ at $x$ to be the order of vanishing of $f_a$ at
$\chi_a(x)$. If $\chi_b$ is a second such chart, then
$\chi_b^*(dz) = \left( \left( \chi_b \circ \chi_a^{-1} \right)' \circ \chi_a \right)\,\chi_a^*(dz)$ and
the derivative does not vanish, so the order of vanishing does not depend on the choice of
chart. Note that the section $\omega$ vanishes at $x$ if and only if the order of $\omega$
at $x$ is positive. We will let $\omega^{-1}(0)$ denote the set of zeros of the 
section $\omega$.

In the following, we will assume that each holomorphic 1-form $\omega$ is not identically zero.
Because $X$ is connected and compact and $\omega$ is holomorphic, the set $\omega^{-1}(0)$ is finite.

If $\phi$ is an orientation-preserving self-homeomorphism of $X$ and $\omega$ is a 1-form that is holomorphic with respect to $J$ represented by the atlas $\{\chi_a : U_a \to \Cbb\}$, 
then the pull-back $\phi^*(\omega)$ is a 1-form that is holomorphic with respect to the structure 
$J \cdot \phi$ represented by the atlas $\{\chi_a \circ \phi : \phi^{-1}(U_a) \to \Cbb\}$.
Moreover, the order of $\omega$ at $x$ equals the order of $\phi^*(\omega)$ at $\phi^{-1}(x)$.
Indeed,
$$
\phi^*(\omega)~
=~
\left( f_a \circ \chi_a \circ \phi \right)\,  (\chi_a \circ \phi)^*(dz),
$$
and so $f_a$ is also the local function associated to $\phi^*(\omega)$ 
in the chart $\chi_a\circ\phi$. 

Without loss of generality, each chart neighborhood $U_a$ is simply connected.
Thus, since $\omega$ is a closed 1-form and $\chi_a$ is a homeomorphism onto its image,
given $x,y \in U_a$, we may unambiguously define
\begin{equation}
\label{eqn:translation-chart}
\nu_x(y)~ 
=~ 
\int_{x}^{y}
\omega
=~
\int_{x}^{y}
\left( f_a \circ \chi_a \right)\, \chi_a^*(dz)~
:=~
\int_{\chi_a(x)}^{\chi_a(y)}
f_a(z)\, dz.
\end{equation}
If $\omega(x) \neq 0$, then $f_a(\chi_a(x)) \neq 0$, 
and hence there exists a neighborhood $U_x$ of $x$ such that
$\nu_x|_{U_x}$ is a biholomorphism from $U_x$ onto an open subset of $\Cbb$. In particular, 
$\omega= \nu_x^*(dz)$.
It follows that the collection $\{\nu_x: U_x \to \Cbb\}$  
is an atlas for $\dot{X}_\omega :=X \setminus  \omega^{-1}(0)$ whose transition functions 
have the form $z \mapsto z+ c$ with $c \in \Cbb$. 
That is, $\nu$ is a 
{\it translation structure} on $\dot{X}_\omega$ \cite{Gutkin-Judge}.

If $\omega(x)=0$ and the order of $\omega$ at $x$ equals $\alpha$, 
then there exists a neighborhood $U$ of $x$ so that
$\nu_x|_{U}$ is a branched covering onto a disk neighborhood of $0$
with ramification point $x$, with branch point $0$, and degree $\alpha+1$.

The holomorphic 1-form $\omega$ determines the underlying complex structure. Indeed, suppose that 
$\chi_a:U_a \to \Cbb$ (resp. $\chi_b':U_b' \to \Cbb$) is a chart for $J$ (resp. $J'$) 
with $x \in U_a \cap U_b'$.  If $\omega(x) \neq 0$, then $\nu_x|_{U_a \cap U_b' \cap U_x}$ 
defined by (\ref{eqn:translation-chart}) is holomorphic with respect to both $J$ and $J'$. 
Hence the map $\chi_a \circ (\chi_b')^{-1}= (\chi_a \circ \nu_x^{-1}) \circ 
(\nu_x \circ (\chi_b')^{-1})$ is holomorphic, and thus  the complex structures $J$ and $J'$ agree 
on $\dot{X}_\omega$. If $\omega(x) = 0$, then Riemann's removable singularity theorem implies 
$\chi_a \circ (\chi_b')^{-1}$ is holomorphic, and so $J=J'$.

Integration of $|\omega|$ defines arclength on $X$, and, in turn, a distance function 
$d_\omega$. If $V \subset  U_x$ is geodesically convex with respect to $d_\omega$, 
then $\nu_x|_{V}$ is an isometry from $V$ onto a convex subset of the Euclidean plane. 
In particular $d_\omega$-geodesics
correspond via $\nu_x$ to line segments in $\Cbb \cong \Rbb^2$.
If $x$ is a zero of $\omega$ of order $\alpha$, then using the branched covering 
described above, one finds that $x$ is 
a conical singularity with cone angle equal to 
$2\pi (\alpha+1)$.

Recall from \S \ref{sec:intro}, the group $\Aff^+_\omega$, the 
homomorphism $D: \Aff^+_\omega \to SL_2(\Rbb)$, and the Veech
group $\Gamma_\omega:= D \Aff^+_\omega$.

\begin{lem}[See \S 2 of \cite{Veech89}]
\label{lem:Veech}
Suppose that $X$ is compact and $\omega$ has at least one zero.
Then 
\begin{itemize}

\item the kernel of $D$ is a finite subgroup,

\item the image, $\Gamma_\omega$, of $D$ is a discrete subgroup of $SL_2(\Rbb)$.
\end{itemize}
\end{lem}

\begin{proof}
Since $\omega$ is holomorphic with at least one zero, the surface $X$ has genus at least two.
If  $\phi \in \ker(D)$, then $\phi$ 
is a self-biholomorphism with respect to the complex structure $J$ that underlies $\omega$.
It is classical that the set of $J$-self-biholomorphisms is finite when 
the genus of the compact surface $X$ is at least two. 

To show discreteness, it suffices to show that if  $\phi_k \in \Aff_\omega^+$ 
is a sequence so that $D \phi_k \to {\rm Id}$, 
then $D\phi_k ={\rm Id}$ for 
all sufficiently large $k$. 
Let $\Tcal$ denote the collection of all oriented triangulations of $X$
whose vertex set is $\omega^{-1}(0)$ and such that 
each 1-cell is a {\em saddle connection}, a $d_\omega$-geodesic with
endpoints in $\omega^{-1}(0)$. The Delaunay triangulation is an example of such a triangulation \cite{Th98}. Because $X$ is a closed surface, 
the number of 1-cells of $T \in \Tcal$ 
equals $3(|\omega^{-1}(0)|- \chi(X))$.
For each $T \in \Tcal$, define $|T|$ to be the maximum length of a 
1-cell in $T$. One can show that the subcollection $\Tcal_0 \subset \Tcal$ of triangulations that minimize $T \to |T|$ 
is finite. 

Because each $1$-cell is a saddle connection, the natural action of the group 
$\Aff^+_\omega$ preserves $\Tcal$. If 
$\phi \in \Aff^+_\omega$ and $T \in \Tcal$, 
then $|\phi \cdot T|\leq \|D\phi\| \cdot |T|$
where $\|D\phi\|$ is the operator norm. 
We claim that for each sufficiently large $k$, the map
$\phi_k$ preserves $\Tcal_0$. Indeed, because $X$ is closed
and $\omega^{-1}(0)$ is finite, one can show that the set of 
lengths of saddle
connections is discrete, and so the image of $T \to |T|$ is discrete.
We have $|\phi_k \cdot T| \leq \|D\phi_k\|\cdot |T|$ and hence 
for each sufficiently large $k$, we have that $\phi_k(T)$
lies in $\Tcal_0$. 

Because $\Tcal_0$ is finite, there exists $m$ so that for each sufficiently 
large $k$ 
we have $\phi_k^m \cdot T =T$ for each oriented triangulation $T \in \Tcal_0$. Fix $T \in \Tcal_0$. Because $T$ has finitely many cells, 
there exists $n \geq m$ so that $\phi^n_k$ acts like the  identity on each oriented 1-cell in $T$.
Since  $\phi_k^n$ is affine,  we find that $\phi_k^n = {\rm id}$, and thus  $D\phi^{n}_k = {\rm Id}$. Hence each $D\phi_k$ is conjugate to a rotation matrix of order 
at most $n$, and therefore the traces of the $D\phi_k$ 
form a discrete set. Because $D \phi_k \to {\rm Id}$
and the identity in $SL_2(\Rbb)$ is the only rotation 
matrix with trace two, 
we have $D \phi_k ={\rm Id}$ for all sufficiently large $k$.

Finally, we show that $\det(D\phi)=1$ for each $\phi \in \Aff^+_\omega$.
A straightforward computation gives
\begin{equation}
\label{eqn:preserve-measure}
\phi^*(\omega \wedge \oomega)~
=~
\phi^*(\omega) \wedge \overline{\phi^*(\omega)}~
=~
(D \phi \cdot \omega ) \wedge \overline{ D\phi \cdot \omega}  ~
=~
\det(D \phi) \cdot \omega \wedge \oomega,
\end{equation}
and hence since $\phi$ is a homeomorphism
$$
\int_X 
\omega \wedge \oomega~
=~
\int_X 
\phi^*(\omega \wedge \oomega)~
=~
\det(D \phi)
\int_X \omega \wedge \oomega.
$$
Therefore $\det(D\phi)=1$ and the image of $D$ lies in $SL_2(\Rbb)$.
\end{proof}

The 2-form $(i/2) \cdot \omega \wedge \oomega$ is the {\it area form} associated to $\omega$.
The associated measure $\mu_\omega$ is defined by 
\begin{equation}
\label{eqn:omega-measure}
\mu_\omega(E)~
:=~
(i/2) \int_E \omega \wedge \oomega.
\end{equation}
From (\ref{eqn:preserve-measure}) we see that each $\phi \in \Aff^+_\omega$
preserves the measure $\mu_\omega$.

\medskip

Define the right action of $\Aff^+_\omega$ on $X$ by $x \cdot \phi := \phi^{-1}(x)$.
If $\phi \in \Aff^+_\omega$, then 
$$
\left(\phi^*(\omega) \right)^{-1}(0)~
=~
(\omega \cdot  \phi)^{-1}(0)~
=~
(D\phi \cdot \omega)^{-1}(0)~
=~
\omega^{-1}(0)
$$
and so each $\phi \in \Aff^+_\omega$ preserves $\omega^{-1}(0)$. 
Moreover, since $\phi$ is a homeomorphism, the order 
of $\omega$ at $x$ equals the order of $\phi^*(\omega)$ at $\phi^{-1}(x)$.

Recall that $\dot{X}_{\omega}:= X \setminus \omega^{-1}(0)$.

\begin{lem}
\label{lem:ker-D-free}
\ \\
If $D\phi= {\rm Id}$ and $\phi(x)=x$ for some $x \in \dot{X}_\omega$,
then $\phi = {\rm id}$.
In other words, $\ker(D)$ acts freely on $\dot{X}_\omega$.
\end{lem}

\begin{proof}
Consider the set $A:=\{y \in \dot{X}_\omega: \phi(y)=y\}$. 
The hypothesis implies that $A$ is nonempty. The set $A$ is closed 
since it is the preimage of the diagonal in $\dot{X}_\omega  \times \dot{X}_\omega$ 
under the map $y \mapsto (y, \phi(y))$.  Suppose that $y \in A$, and define
the chart $\nu: U \to \Cbb$ as in (\ref{eqn:translation-chart}) 
$$
\nu({y'})~ 
:=~
\int_y^{y'} \omega
$$
where $U$ is simply connected.

Since $\phi(y)=y$ and $\phi$ is continuous,
there exists a neighborhood $V$ of $y$ so that $V \subset U$ and $\phi(V) \subset U$. Because  $\phi(y)=y$
and $\omega = D \phi \cdot \omega = \omega \cdot \phi$, we have for each ${y'} \in V$
$$
\nu(\phi({y'}))~
=~
\int_{y}^{\phi({y'})} \omega~
=~
\int_{\phi(y)}^{\phi({y'})} \omega~
=~
\int_y^{y'} \phi^*(\omega)~
=~
\int_y^{y'}  \omega \cdot \phi~
=~
\int_y^{y'}  \omega~
=~
\nu({y'}).
$$
Since $\nu$ is injective, we have $\phi({y'})={y'}$ and so each ${y'} \in V$ belongs 
to $A$. Hence $A$ is open. 

In sum,  $A$ is a nonempty, closed and  open subset of 
a connected space $\dot{X}_\omega$. Thus $A=\dot{X}_\omega$ and 
since $\dot{X}_\omega$ is dense in $X$, the continuous map $\phi$ equals the 
identity. 
\end{proof}

Since the finite subgroup $\ker(D)$ acts freely on the right of $\dot{X}_{\omega}$,
the quotient 
$
X^*_{\omega} ~:=~ \dot{X}_{\omega}/\ker(D)
$
is a punctured surface. Since $\tau^*(\omega)=\omega$ for each $\tau \in \ker(D)$, 
the measure $\mu_\omega$ on $X$ pushes forward to a finite measure $\mu^*$ on $X^*_{\omega}$.

Write $[x]$ for the image of $x \in \dot{X}_{\omega}$ in $X^*_{\omega}$.

\begin{lem}
\label{lem:gamma-acts}
The Veech group $\Gamma_\omega$ acts on the left of $X^*_{\omega}$ by
$\gamma \cdot [x] := \left[ \phi(x) \right]$, where $\phi \in \Aff^+_\omega$
is an affine map with $D\phi = \gamma$.  This action is well-defined,
faithful, and preserves the measure $\mu^*$.
\end{lem}

\begin{proof}
If $D\phi = D\psi = \gamma$, then $\phi \circ \psi^{-1} \in \ker(D)$, so $\left[ \phi(x) \right] = \left[ \psi(x) \right]$. 

If $y = \tau(x)$ with $\tau \in \ker(D)$,  then since $\ker(D)$ is normal in $\Aff_\omega^+$, we have $\phi \circ \tau = \tau' \circ \phi$ for some $\tau' \in \ker(D)$. Therefore,
$\left[ \phi(y) \right] = \left[ \tau'\left( \phi(x) \right) \right] = \left[ \phi(x) \right]$.

To prove faithfulness, suppose that $\gamma$ acts like the identity, 
and fix $\phi$ so that $D\phi =\gamma$. Then by the definition of the action, 
for each $x \in \dot{X}_{\omega}$ there exists $\tau_x$ so  that $\phi(x)=\tau_x (x)$. 
In other words, if for each $\tau \in \ker(D)$, we define $
A_\tau :=\left\{ x \in \dot{X}_{\omega} : \phi(x) = \tau(x)\right\}$, then the union of the sets $A_\tau$ over $\tau \in \ker(D)$ covers $\dot{X}_\omega$.
Each set $A_\tau$ is closed because it is the fixed point set of $\tau^{-1} \circ \phi$. 
The sets $A_\tau$ are also disjoint and finite in number. Hence each $A_\tau$ is also open.
Since $\dot{X}_{\omega}$ is connected, exactly one $A_\tau$ equals $\dot{X}_{\omega}$. In other words,
$\phi = \tau$ for a single $\tau \in \ker(D)$ and so $\gamma = D\phi = I$.

To prove that $\gamma \in \Gamma_\omega$ is measure-preserving, 
fix $\phi \in \Aff^+_\omega$ so that
$D\phi = \gamma$. Let $\pi: \dot{X}_{\omega} \to X^*_{\omega}$ denote the projection. Then for any Borel set $E \subseteq X^*_{\omega}$,
\begin{eqnarray*}
\mu^*\left( \gamma^{-1} \cdot E \right)~
&=&~
\mu_\omega\left( \pi^{-1}\left( \gamma^{-1} \cdot E \right) \right)
\\
&=&~
\mu_\omega\left( \phi^{-1}\left( \pi^{-1}(E) \right) \right)
\\
&=&~
\mu_\omega\left( \pi^{-1}(E) \right)~
=~
\mu^*(E),
\end{eqnarray*}
where the second equality holds because
$\pi \circ \phi = \gamma \cdot \pi$ by the definition
of the action, and the third because $\phi$ preserves
$\mu_\omega$ by (\ref{eqn:preserve-measure}).  Hence $\mu^*$ is
$\Gamma_\omega$-invariant.
\end{proof}

%

%%%%%%%%%%%%%%%%%%%%%%%%%%%%%%%%%%%%
%%%%%%%%%%%%%%%%%%%%%%%%%%%%%%%%%%%%
%%%%%%%%%%%%%%%%%%%%%%%%%%%%%%%%%%%%
%%%%%%%%%%%%%%%%%%%%%%%%%%%%%%%%%%%%

\ \medskip

\section{The spectral gap of Avila and Gou\"ezel}
\label{sec:AG}

\ \medskip

In this section we recall the main result of \cite{Avila-Gouezel-13-small-eigenvalues} and 
discuss its generalization to the context of fake zeros. (See also Appendix \ref{app:avila-gouezel}).

First, we provide some background concerning strata of holomorphic 1-forms and the $GL_2^+(\Rbb)$-action on each stratum. Let $X$ be a connected, compact, oriented topological surface without boundary, let 
$\Sigma=\{\sigma_1, \ldots,\sigma_d\}$ be  a finite subset of $X$, and let 
$\alpha$ be a $d$-tuple of nonnegative integers.
Let $\Omega(X,\Sigma, \alpha)$ denote the set of pairs $(J,\omega)$ where 
$J$ is a complex structure on $X$ compatible with the orientation of $X$
and $\omega$ is a complex valued 1-form with the following properties:
\begin{itemize}

\item $\omega$ is holomorphic with respect to $J$;

\item the zero set $\omega^{-1}(0)$ is contained in $\Sigma$;

\item if $\omega(\sigma_i)=0$, then the zero $\sigma_i$ has order $\alpha_i>0$;

\item if $\omega(\sigma_i)\neq 0$, then $\alpha_i=0$. In this latter case 
      we refer to $\sigma_i$ as a {\it marked point} or a {\em fake zero};

\item  $\omega$ is not identically zero.

\end{itemize}

Let $\Homeo^+(X, \Sigma,\alpha)$ denote the group of orientation-preserving 
self-homeomorphisms
$\phi:X \to X$ so that  $\phi(\Sigma)=\Sigma$ and so that if
$\phi(\sigma_i)=\sigma_j$ then $\alpha_i=\alpha_j$. The group 
$\Homeo^+(X, \Sigma,\alpha)$ acts naturally on the right of 
$\Omega(X, \Sigma,\alpha)$: if $\phi \in \Homeo^+(X, \Sigma,\alpha)$
and $\{\chi_a\}$ is an atlas that represents
$J$ then $J \cdot \phi$ is the complex structure associated to 
$\{\chi_a \circ \phi\}$. Hence $\phi$ is a biholomorphism 
from $(X,J\cdot \phi)$ to $(X,J)$. Define 
$$
\omega \cdot \phi~ :=~ \omega \circ d \phi~ =~ \phi^{*}(\omega).
$$
Thus we have a right action defined by 
$(J, \omega)\cdot \phi :=(J \cdot \phi, \omega \cdot \phi)$.
Given that the holomorphic 1-form $\omega$ determines its complex 
structure $J$---see \S \ref{sec:1forms-and-affine-diffeos}---we will usually 
suppress $J$, and we will often regard  $\Omega(X, \Sigma,\alpha)$ as 
simply the set of 1-forms that are holomorphic with respect to some 
complex structure.

Let $\Homeo^0(X,\Sigma)$ denote the identity component  
of $\Homeo^+(X, \Sigma,\alpha)$. 
In other words, $\Homeo^0(X, \Sigma)$ consists 
of homeomorphisms that are isotopic to the identity 
through an isotopy consisting of homeomorphisms that 
fix each point in $\Sigma$.
Define\footnote{
In \cite{Avila-Gouezel-13-small-eigenvalues}, the authors
employ the equivalent language of translation structures. See \S \ref{sec:1forms-and-affine-diffeos}
here for the well-known `dictionary'. 
}
$$
\Teich (X,\Sigma, \alpha)~
:=~
\Omega(X,\Sigma, \alpha)/\Homeo^0(X,\Sigma).
$$
The set $\Teich (X,\Sigma, \alpha)$ has a natural complex
affine manifold structure given by {\it period coordinates.}
Choose a basis $\gamma_1, \ldots, \gamma_r$
for the relative homology group  $H_1(X,\Sigma;\Zbb)$. 
Then there is a neighborhood $U_{\omega_0,\gamma}$ 
of $[\omega_0] \in \Teich(X, \Sigma, \alpha)$ so that the map
$\Per_{\omega_0,\gamma}: U_{\omega_0,\gamma} \to \Cbb^r$ defined by
\begin{equation}
\label{eqn:period-coordinates}
\Per_{\omega_0,\gamma}\left(\omega\, \Homeo^0(X, \Sigma) \right)~
:=~
 \left( \int_{\gamma_1} \omega, \ldots, \int_{\gamma_r} \omega \right)
\end{equation}
is a bijective map
from $U_{\omega_0,\gamma}$ onto an open subset of $\Cbb^r$ where $r= {\rm rank}(H_1(X,\Sigma;\Zbb))=2g+d-1$.
These period coordinate maps constitute an atlas for $\Teich(X, \Sigma, \alpha)$
and the transition functions are complex affine maps. 
(See for example \cite{Forni-Matheus}.)

The group
$$
\Mod(X,\Sigma,\alpha)~
:=~
\Homeo^+(X,\Sigma, \alpha)/\Homeo^0(X,\Sigma)
$$
acts properly discontinuously 
on $\Teich(X,\Sigma,\alpha)$ as complex affine homeomorphisms.
The quotient orbifold  $\Mcal(X,\Sigma,\alpha)$ is an affine complex orbifold 
called the {\em stratum} associated to $\alpha$.

Note that the area 
of a 1-form is unchanged by the pull-back action of $\Homeo^+(X, \Sigma, \alpha)$. 
We will use a superscript $1$ to indicate a collection of equivalence 
classes of 1-forms that have area 
equal to 1. For example, $\Mcal^1(X,\Sigma,\alpha)$ is the subspace of 
$\Mcal(X,\Sigma,\alpha)$ that consists of equivalence
classes of 1-forms which have area equal to 1.

Let $GL_2^+(\Rbb)$ denote the $2 \times 2$ matrices with positive determinant.
We have a natural left action of $GL_2^+(\Rbb)$ on $\Cbb$ 
given by 
\begin{equation}
\label{eqn:g-action}
\begin{pmatrix}
g_{11} & g_{12} \\
g_{21} & g_{22}
\end{pmatrix}
\cdot
(x+ iy)~
=~
g_{11} x + g_{12} y~  +~  i(g_{21} x + g_{22} y)
\end{equation}
where $i = \sqrt{-1}$.

This left action on $\Cbb$ induces a left action on $\Omega(X,\Sigma,\alpha)$ as follows:
Let $\{\nu_x : U_x \to \Cbb\}$ be the translation structure atlas on $X\setminus \Sigma$
defined in (\ref{eqn:translation-chart}). Then since the translation group is a normal
subgroup of the group of real affine transformations, the transition functions associated
to the atlas $\{ g \circ \nu_x : U_x \to \Cbb\}$ are also translations, and hence
this atlas defines a complex structure on $X \setminus \Sigma$. If $\sigma_i \in \Sigma$ has order 
$\alpha_i$, then $g \circ \nu_{\sigma_i}$
is a degree $\alpha_i+1$ topological branched covering branched over $0$ with ramification point 
$\sigma_i$. There exists a neighborhood 
$V_i$ of $\sigma_i$ and an embedding  $\eta_i: V_i \to \Cbb$ so that $y \mapsto \eta_i(y)^{\alpha_i+1}$ coincides with $g \circ \nu_{\sigma_i}$ 
on the intersection of their domains.  
Note that on $U_x \cap V_i$, we have $d(\nu_x -\nu_{\sigma_i})=0$  
and hence $\nu_x -\nu_{\sigma_i}$ 
is constant on each component of $U_x \cap V_i$.
Hence $g \circ \nu_x =(g \circ \nu_{\sigma_i}) + g \cdot c$, and so
$$
    (g \circ \nu_x) \circ \eta_i^{-1}(z)~ =~ z^{\alpha_i+1} + g \cdot c
$$
is holomorphic on the overlap and has nonzero derivative there.
Therefore by adding the collection 
$\{ \eta_i: V_i \to \Cbb\}$ to the translation atlas $\{ g \circ \nu_x : U_x \to \Cbb\}$,
we obtain a complex structure $g \cdot J$ for $X$. If $\omega$ is a 1-form that is 
holomorphic with respect to $J$, then define a $g \cdot J$ holomorphic 1-form on $X$ 
by setting $g \cdot \omega|_{U_x}= (g \circ \nu_x)^{*}(dz)$
and $g \cdot \omega|_{V_i}= \eta_i^{*}(d(z^{\alpha_i+1}))$.
This left action of $GL_2^+(\Rbb)$ on $\Omega(X, \Sigma, \alpha)$
commutes with the right action of $\Homeo^+(X, \Sigma, \alpha)$ on 
$\Omega(X,\Sigma,\alpha)$, and hence it descends to left-actions of 
$GL_2^+(\Rbb)$ on both $\Teich(X,\Sigma,\alpha)$ and $\Mcal(X,\Sigma,\alpha)$.

An element $g \in GL_2^+(\Rbb)$ preserves the area of $\omega$ iff $\det(g)=1$,  
and hence $G:=SL_2(\Rbb)$ acts on $\Mcal^1(X,\Sigma,\alpha)$.
Let $\mu$ be a measure on $\Mcal^1(X,\Sigma,\alpha)$ that is 
invariant under the action of $G$: If $E$ is $\mu$-measurable, then 
$\mu(g \cdot E) = \mu(E)$.  
The Koopman representation $\pi_\mu$
of $G=SL_2(\Rbb)$ on the Hilbert space 
$L^2\left(\Mcal^1(X,\Sigma, \alpha), \mu \right)$
is defined by 
\begin{equation}
\label{eqn:repn-on-moduli}
\pi_\mu(g)\, u([\omega])~
:=~
u(g^{-1} \cdot[\omega]).
\end{equation}
Because $\mu$ is $G$-invariant, the representation $\pi_\mu$ is unitary.
Thus, the representation $\pi_\mu$ admits a direct integral
decomposition into irreducible unitary representations.

The main result of \cite{Avila-Gouezel-13-small-eigenvalues} concerns the irreducible representations 
of $\pi_\mu$ that are isomorphic to `complementary series' representations of 
$SL_2(\Rbb)$. The complementary series is one of the 
classes of irreducible `spherical' representations, the representations of 
$SL_2(\Rbb)$ that have a vector fixed by $K:=SO(2)$.
The set of isomorphism classes
of the spherical representations of $SL_2(\Rbb)$ consists of
the trivial representation, the  spherical principal series, and the complementary series.

There are several explicit models of the irreducible spherical representations. 
These are usually described using a parameter `$s$', but unfortunately the meaning 
of `$s$' depends on the model. In order to fix the meaning of `$s$' we choose 
the following explicit models given in \S 3.1 and \S 3.2 of \cite{GGP}.\footnote{Note that we use the models of \cite{GGP} with $s$ replaced by $-s$
so that the complementary series corresponds to $0< s <1$.}

For each $s \in \Cbb$, we let $\Scal^+_s$ 
denote the space of smooth functions $f:\Rbb^2 \setminus \{0\} \to \Cbb$ such that
\begin{itemize}

\item $f(-x)=f(x)$

\item $f(t \cdot x)= t^{-s-1} \cdot f(x)$ for $t>0$.
\end{itemize}

For $s \in \Cbb$ such that ${\rm Re}(s)=0$ or $0 < s < 1$, define on
$\Scal^+_s$ the following inner product
$$
\langle u, v \rangle_s~
=~
\left\{
\begin{array}{ll}
 \displaystyle\int_{|x|=1} u(x) \cdot \ov(x)\, d \sigma(x) &\mbox { if } {\rm Re}(s)=0 \\  [1em]
 \displaystyle\int_{|x|=1} 
 \displaystyle\int_{|y|=1} 
  u(x) \cdot \ov(y)\,    
K_s(x, y)\, d \sigma(y)\, d \sigma(x) 
 & \mbox { if }  0 < s < 1,
\end{array}
\right.
$$
where $d \sigma(x)$ is arclength measure on the circle  $|x|=1$ and 
$$
K_s(x,y)~ 
:=~
|x_1 y_2 - y_1 x_2|^{s-1}.
$$
Let $\Hcal^+_s$ denote the completion of $\Scal^+_s$ 
with respect to the associated norm $\|\cdot \|_s$.
The group $SL_2(\Rbb)$ acts on the left of $\Rbb^2$ by matrix multiplication.
Define the representation $\pi_s$ on  $\Hcal_s^{+}$ by
$$
\pi_s(g) u(x)~
=~
u( g^{-1} \cdot x).
$$
If ${\rm Re}(s)=0$, then $\pi_s$ is called 
a {\it spherical principal series representation}.
If $0 < s <1 $, then $\pi_s$ is called a 
{\it complementary series representation}.
One should regard $\pi_1$ as the trivial representation.

The theorem of Avila and Gou\"ezel provides a `gap' between the trivial 
representation and the other spherical irreducible representations that appear in $\pi_\mu$.

\begin{restatable}{thm}{thmAG}[compare Theorem 2.3 \cite{Avila-Gouezel-13-small-eigenvalues}]
\label{thm:AG}
\ \\
Let $\mu$ be an ergodic $SL_2(\Rbb)$-invariant finite measure on $\Mcal^1(X,\Sigma, \alpha)$
and let $\pi$ be the associated unitary
representation of $SL_2(\Rbb)$ on $L^2(\Mcal^1(X,\Sigma, \alpha), \mu)$
defined in (\ref{eqn:repn-on-moduli}).
Then the set of $s \in (0,1]$ 
such that $\pi_s$ appears in the decomposition of $\pi$ 
is discrete in $(0,1]$.  In particular, the supremum of $s$ in $[0,1)$ such 
that $\pi_s$ appears in the decomposition of $\pi$ is strictly less than 1.
\end{restatable}

\begin{proof} 
The principal difference between this theorem and 
Theorem 2.3 in \cite{Avila-Gouezel-13-small-eigenvalues} is that 
we allow here for finitely many fake zeros. We claim 
that the proof of Theorem \ref{thm:AG} is 
nearly identical to the proof of 
Theorem 2.3 in \cite{Avila-Gouezel-13-small-eigenvalues}. 
We validate this claim in Appendix \ref{app:avila-gouezel}.
\end{proof}

Several remarks should be made. 

First, Theorem 2.3 in \cite{Avila-Gouezel-13-small-eigenvalues}
is stated in terms of a foliated Laplacian (when each $\alpha_i>0$).  
Immediately after the statement 
of Theorem 2.3 \cite{Avila-Gouezel-13-small-eigenvalues},
the authors provide  an alternative formulation that is equivalent to 
Theorem \ref{thm:AG} above (when each $\alpha_i>0$). Note that 
the proof of Theorem 2.3 in \S 3.6 of 
\cite{Avila-Gouezel-13-small-eigenvalues} is conducted entirely within 
the representation theoretic context, resting on the dynamical 
properties of the Lie derivative associated to the Teichm\"uller geodesic flow.
Indeed, in \S 3.6 of \cite{Avila-Gouezel-13-small-eigenvalues},
the authors prove Theorem \ref{thm:AG} above. Theorem 2.3 
in \cite{Avila-Gouezel-13-small-eigenvalues} then follows
using the equivalence described in \S 3.4 of \cite{Avila-Gouezel-13-small-eigenvalues}. The Laplacian logically enters only after 
the representation theoretic statement has been proven.

Second, Theorem \ref{thm:AG} was first proven in \cite{AGY-06-exp-mix} in the special case 
that $\mu$ lies in the natural Lebesgue measure class of a given connected 
component of a stratum. 

Third, the statement of Theorem 2.3 in \cite{Avila-Gouezel-13-small-eigenvalues} concerns so-called `admissible measures' $\widetilde{\mu}$ on  $\Mcal^1(X, \Sigma, \alpha)$. 
But as the authors point out just after the definition of admissible measure, 
a result of Eskin and Mirzakhani \cite{EM} shows that each ergodic $SL_2(\Rbb)$-invariant 
measure is admissible. Indeed, Theorem 1.4 in \cite{EM} states 
that each $P$-ergodic probability measure on $\Mcal^1(X, \Sigma, \alpha)$ is $SL_2(\Rbb)$-invariant and 
affine.\footnote{Here $P$ is the subgroup of upper triangular
matrices in $SL_2(\Rbb)$.} 
Note that $G$-ergodicity implies $P$-ergodicity. Indeed, suppose that $E$ is a $P$-invariant
measurable set and define measures $\nu$ and $\nu^c$ by $\nu(A) = \tmu(A \cap E)/\tmu(E)$ and 
$\nu^c(A) = \tmu(A\cap E^c)/\tmu(E^c)$ respectively. Then $\nu$ and $\nu^c$ are 
$P$-invariant probability measures, and hence can be decomposed into $P$-ergodic 
measures. Theorem 1.4 in \cite{EM} implies that each such measure is $G$-invariant,
and hence $\nu$ and $\nu^c$ are both $G$-invariant.
Note that $\tmu = \tmu(E) \nu + \tmu(E^c) \nu^c$, and hence since $\tmu$
is $G$-ergodic, it follows that $\tmu(E)=0$ or $\tmu(E^c)=0$. Thus, $\tmu$ is $P$-ergodic.

\newpage
%%%%%%%%%%%%%%%%%%%%%%%%%%%%%%%%%%%%%%%%%%%%%%%%%%%%%%%%%

\ \medskip

\section{The bundle $Y_\omega$ and an $SL_2(\Rbb)$-equivariant embedding into a stratum}
\label{sec:aff-manifold}

\ \medskip

In this section, we consider the space $Y_\omega$ that we introduced in \S \ref{sec:intro} 
and show that it is diffeomorphic to the total space of a bundle of the type 
considered in \S \ref{sec:duality}. Then we show that $Y_\omega$ 
embeds $SL_2(\Rbb)$-equivariantly into a stratum  $\Mcal^1(X,\Sigma, \alpha)$.
At the end of the section, we show that the pushforward of the natural measure $\nu$ on $Y_\omega$ is an ergodic affine measure
and thus conclude from Theorem \ref{thm:AG} that 
the Koopman representation $\pi_\omega$ on $L^2(Y_\omega, \nu)$ has a spectral gap.

Fix a nontrivial complex-valued 1-form $\omega$ on $X$
that is holomorphic with respect to some complex structure and has area equal to 1.
The group $\Aff^+_\omega$ acts on $GL_2^+(\Rbb) \times X$ by 
$$
\left(g, x \right) \cdot \phi~
=~ 
   (g \cdot D\phi, \phi^{-1}(x)).
$$
Recall that $\dot{X}_{\omega} := X \setminus \omega^{-1}(0)$. 
We will assume that the group $\Gamma_\omega$ is a lattice.

\begin{lem}
\label{lem:Aff-discontinuous}
The action of $\Aff^+_\omega$ on $GL_2^+(\Rbb) \times \dot{X}_\omega$ is free
and properly discontinuous. In particular, 
the associated quotient map is a smooth covering map.
\end{lem}

\begin{proof}
If $(g \cdot D\phi, \phi^{-1}(x))= (g,x)$, then $D\phi={\rm Id}$ and $\phi^{-1}(x)=x$.
Thus, from Lemma \ref{lem:ker-D-free} we find that $\phi ={\rm id}$.
Hence the action is free. To see that the action is properly discontinuous, 
suppose that $K \subset GL_2^+(\Rbb) \times \dot{X}_{\omega}$ is compact. If 
there exists $\phi \in \Aff^+_\omega$ and $(g,x) \in K$ so that $(g,x) \cdot \phi \in K$,
and $p_1$ denotes the projection onto $GL_2^+(\Rbb)$, then 
$\phi$ belongs to the set of $\phi$ such that there exists $g \in p_1(K)$
such that $g\, D\phi \in p_1(K)$. The set $p_1(K)$ is compact, and, by Lemma 
\ref{lem:Veech}, the group $\Gamma_\omega := D \Aff^+_\omega$ is a discrete
subgroup of the Lie group $GL^+_2(\Rbb)$. Hence $\Gamma_\omega$ 
acts properly discontinuously on $GL^+_2(\Rbb)$,  and so the set 
$\{D \phi\, :\, p_1(K)\cdot D \phi \cap p_1(K) \neq \emptyset \}$ is finite.
Each fiber of $D$ has cardinality $|\ker(D)|$ which is finite by Lemma \ref{lem:Veech}. It follows that 
$\{\phi\, :\, K \cdot \phi \cap K \neq \emptyset \}$ is finite and hence 
$\Aff^+_\omega$ acts properly discontinuously on $GL_2^+(\Rbb) \times \dot{X}_\omega$.
\end{proof}

Because $\Gamma_\omega \subset G:=SL_2(\Rbb)$, the 
submanifold $G \times \dot{X}_{\omega}$ is preserved by the action
of $\Aff^+_\omega$. The quotient $Y_\omega:= (G \times \dot{X}_{\omega})/\Aff^+_\omega$
is the smooth manifold introduced in \S \ref{sec:intro}.
The $\Aff^+_\omega$-action preserves the product measure $ m \times \mu_\omega$
where $m$ is a Haar measure on $G$. Hence the measure $ m \times \mu_\omega$ induces 
a measure $\nu$ on $Y_\omega$.

The manifold $Y_\omega$ is naturally a fiber bundle over $G/\Gamma_\omega$
with fiber $X^*_{\omega} := \dot{X}_{\omega} / \ker(D)$.  To realize this 
bundle structure, recall from Lemma~\ref{lem:gamma-acts} that the group 
$\Gamma_\omega$ acts on the left of
$X^*_{\omega}$ by $\gamma \cdot [x] = \left[ \phi(x) \right]$ for any
$\phi \in \Aff^+_\omega$ with $D\phi = \gamma$. Thus, $\Gamma_\omega$ acts on 
the right of $G \times X^*_{\omega}$ via
\begin{equation}
\label{eqn:bundle-construct}
\left( g, [x] \right) \cdot \gamma~
:=~
\left( g \cdot \gamma,\ \gamma^{-1} \cdot [x] \right)~
=~
\left( g \cdot \gamma,\  [\phi^{-1}(x)] \right).
\end{equation} 
An argument similar to the proof of Lemma \ref{lem:Aff-discontinuous}
shows that the action is free and properly discontinuous, and hence 
$G \times X^*_{\omega} \to (G \times X^*_{\omega})/\Gamma_\omega$
is a covering map.

The projection $G \times X^*_\omega \to G$ descends to a surjective map 
$(G \times X^*_\omega)/\Gamma_\omega \to G/\Gamma_\omega$ making 
$(G \times X^*_\omega)/\Gamma_\omega \to G/\Gamma_\omega$ a fiber bundle
with fiber equal to $X^*_{\omega}$. 

\begin{prop}
The map $(g,x)\Aff^+_\omega \mapsto (g, [x])\Gamma_\omega$ is a
diffeomorphism from $Y_\omega$ onto $(G \times X^*_\omega)/\Gamma_\omega$
that intertwines the $G$-actions.
\end{prop}

\begin{proof}
We first show that the map is well-defined. Suppose $\phi \in \Aff^+_\omega$. By
Lemma \ref{lem:gamma-acts} we have 
$\left[ \phi^{-1}(x) \right] = \left( D\phi \right)^{-1} \cdot [x]$,
and hence
$$
\left( g\, D\phi, \left[ \phi^{-1}(x) \right] \right)~
=~
\left( g\, D\phi, (D \phi)^{-1} \cdot [x] \right)~
=~
 \left( g, [x] \right) \cdot D\phi.
$$
Thus the map is well-defined.

Surjectivity of the map follows from the surjectivity of $x \mapsto [x]$.

To prove that the map is injective, suppose that
$(g, [x])\Gamma_\omega = (h, [y])\Gamma_\omega$. Then there exists
$\gamma \in \Gamma_\omega$ so that
$$
\left( h, [y] \right)~ =~ \left( g, [x] \right) \cdot \gamma~ =~ \left( g\gamma,\ \gamma^{-1} \cdot [x] \right).
$$
Choose $\phi \in \Aff^+_\omega$ with $D\phi = \gamma$. Then
$\gamma^{-1}\cdot [x] = \left[ \phi^{-1}(x) \right]$, and so $[y] = \left[ \phi^{-1}(x) \right]$.
Hence there exists $\tau \in \ker(D)$ so that $y = \tau^{-1}\left( \phi^{-1}(x) \right)$.
Set $\psi := \phi \circ \tau$. Since $D$ is a homomorphism and $D\tau = {\rm Id}$, we have
$D\psi = \gamma$, and by construction $\psi^{-1}(x)=y$. Therefore
$$
(g,x) \cdot \psi~ =~ \left( g\, D\psi,\ \psi^{-1}(x) \right)~ =~ \left( g\gamma, y \right)~ =~ (h,y),
$$
and hence $(g,x)\Aff^+_\omega =(h,y)\Aff^+_\omega$. 

The map $x \mapsto [x]$ is a smooth covering map, and hence 
the map $(g,x) \mapsto (g, [x])$ is a smooth covering map. 
The covering maps $G \times \dot{X}_\omega \to Y_\omega$ and  $G \times X^*_{\omega} \to (G \times X^*_{\omega})/\Gamma_\omega$ are also smooth. Since smooth covering maps are local diffeomorphisms, it follows that the map
is a local diffeomorphism. Since this map is bijective, it is a global diffeomorphism.

For each $h \in G$ we have $ h \cdot (g,x)\Aff^+_\omega = (hg,x)\Aff^+_\omega 
\mapsto (h g, [x])\Gamma_\omega= h \cdot (g, [x])\Gamma_\omega$,
and hence the map intertwines the left $G$-actions.
\end{proof}

The goal of the remainder of the section is to show how $Y_\omega$ 
naturally embeds into a stratum in a $G$-equivariant fashion.
To describe the stratum into which $Y_\omega$ embeds, we first
enumerate the zeros of $\omega$ as $\sigma_1, \ldots, \sigma_{d-1}$, 
choose some $\sigma_d \in \dot{X}_\omega$ and set 
$\Sigma := \{\sigma_1, \ldots, \sigma_d\}$.
For $i< d$ let $\alpha_i$ be the order of the zero $\sigma_i$ of $\omega$,
and let $\alpha_d =0$.  With these choices, note that 
$\omega \in \Omega^1(X, \Sigma, \alpha)$.
This is the stratum into which we will embed $Y_\omega$.

Before defining the embedding we first observe that the stratum 
$\Mcal(X, \Sigma, \alpha)$ is a smooth manifold. Indeed, any 
stratum with exactly one marked point is smooth.

\begin{lem}
\label{lem:free-action}
Suppose that $\alpha_d = 0$ and $\alpha_i > 0$ for each $i < d$.
If $\eta \in \Omega(X,\Sigma, \alpha)$ and $\phi \in \Homeo^+(X,\Sigma, \alpha)$ satisfies
$\left( \eta\, \Homeo^0(X,\Sigma) \right) \cdot \phi = \eta\, \Homeo^0(X,\Sigma)$,
then $\phi \in \Homeo^0(X,\Sigma)$.
In particular, the group $\Mod(X,\Sigma,\alpha) = \Homeo^+(X,\Sigma,\alpha)/\Homeo^0(X,\Sigma)$
acts freely on $\Teich(X,\Sigma,\alpha)$, and hence
$\Mcal(X,\Sigma,\alpha)$ is a manifold.
\end{lem}

\begin{proof}
Suppose that
$\left( \eta\, \Homeo^0(X,\Sigma) \right) \cdot \phi = \eta\, \Homeo^0(X,\Sigma)$.
Then there exists $\psi \in \Homeo^0(X,\Sigma)$ so that $\eta \cdot \phi = \eta \cdot \psi$.
Set $\chi := \phi \circ \psi^{-1}$. Because the action is a right action, we have
$$
\eta \cdot \chi~
=~
\left( \eta \cdot \phi \right)\cdot \psi^{-1}~
=~
\left( \eta \cdot \psi \right) \cdot \psi^{-1}~
=~
\eta .
$$
In other words, $\chi$ is an affine homeomorphism of $(X,\eta)$ whose
differential is the identity.

Since $\phi$ and $\psi$ both belong to $\Homeo^+(X,\Sigma,\alpha)$, so does $\chi$.
By hypothesis $\alpha_d = 0$ and $\alpha_i > 0$ for each $i<d$, and hence $\sigma_d$
is the unique point of $\Sigma$ whose associated order is zero. Since $\chi$
preserves these orders, we have $\chi(\sigma_d) = \sigma_d$. Moreover,
$\eta(\sigma_d) \neq 0$ and so $\sigma_d \in \dot{X}_\eta$.

Thus $\chi$ fixes a point of $\dot{X}_\eta$ and has differential equal to the
identity. Lemma \ref{lem:ker-D-free}, applied to the 1-form $\eta$, implies that
$\chi = {\rm id}$, and therefore $\phi = \psi \in \Homeo^0(X,\Sigma)$.

It follows that the stabilizer in $\Mod(X,\Sigma,\alpha)$ of each point of
$\Teich(X,\Sigma,\alpha)$ is trivial, that is, the action of
$\Mod(X,\Sigma,\alpha)$ on $\Teich(X,\Sigma,\alpha)$ is free. Since, for general
$\alpha$, this action is well-known to be properly discontinuous, the quotient
$\Mcal(X,\Sigma,\alpha)$ is a manifold.
\end{proof}

We next define the $G$-equivariant 
embedding $\Phi_\omega: Y_\omega \to \Mcal^1(X, \Sigma, \alpha)$.
Given $(g,x)  \in SL_2(\Rbb) \times \dot{X}_\omega$, let
$\theta$ be an orientation-preserving  self-homeomorphism of $X$ that is  
such that $\theta$ fixes $\sigma_j$ if $j<d$ and such that 
$\theta(\sigma_d)= x$.\footnote{Since $\dot{X}_\omega$ is connected,
such a homeomorphism exists.  For example, one can use a `point-pushing'
homeomorphism \cite{Farb-Margalit}.}
Then $\omega \cdot \theta = \theta^*(\omega) \in \Omega(X, \Sigma, \alpha)$
and we define 
\begin{equation}
\label{eqn:Phi-defn}
\Phi_\omega\left((g,x) \Aff^+_\omega\right)~
:=~
(g \cdot \omega \cdot \theta)\, \Homeo^+(X,\Sigma, \alpha).
\end{equation}

\begin{lem}
The map $\Phi_\omega$ is well-defined and does not depend on the 
choice of $\theta$.
\end{lem}

\begin{proof}
We first note that the right hand side of (\ref{eqn:Phi-defn})
does not depend on the choice of $\theta$.
Indeed, if $\psi$ is an orientation-preserving self-homeomorphism of $X$ 
that 
also fixes $\sigma_i$ for each $i<d$
and that maps $\sigma_d$ to $x$, then
$\theta^{-1} \circ \psi$
lies in $\Homeo^+(X,\Sigma, \alpha)$. 
Thus 
$$
(\omega \cdot \psi)\, \Homeo^+(X,\Sigma, \alpha)~
=~ 
(\omega \cdot \theta \circ \theta^{-1} 
\circ \psi)\,  \Homeo^+(X,\Sigma, \alpha)~
=~
(\omega \cdot \theta)\, \Homeo^+(X,\Sigma, \alpha).
$$
Therefore, the right hand side of (\ref{eqn:Phi-defn})
does not depend on the choice of $\theta$.

Now suppose that $\phi \in \Aff^+_\omega$, and recall that
$(g,x)\cdot \phi = \left( g\, D\phi,\ \phi^{-1}(x) \right)$.
Let $\theta'$ be an orientation-preserving self-homeomorphism of $X$ such that
$\theta'$ fixes $\sigma_j$ for each $j<d$ and such that
$\theta'(\sigma_d) = \phi^{-1}(x)$; that is, $\theta'$ is a choice associated to
the point $\phi^{-1}(x)$ in the same way that $\theta$ is a choice associated to $x$. Set
$$
\chi~ :=~ \theta^{-1} \circ \phi \circ \theta'.
$$
We claim that $\chi \in \Homeo^+(X,\Sigma,\alpha)$. Indeed, $\chi$ is orientation
preserving, and
$$
\chi(\sigma_d)~
=~
\theta^{-1}\left( \phi \left( \phi^{-1}(x) \right) \right)~
=~
\theta^{-1}(x)~
=~
\sigma_d .
$$
Moreover, if $j<d$, then $\theta'(\sigma_j) = \sigma_j$, and since $\phi$ is affine
it permutes $\omega^{-1}(0) = \{\sigma_1, \ldots, \sigma_{d-1}\}$ and preserves the
order of each zero. Hence $\phi(\sigma_j)= \sigma_k$ for some $k<d$ with
$\alpha_j = \alpha_k$, and since $\theta^{-1}$ fixes $\sigma_k$ we find that
$\chi(\sigma_j) = \sigma_k$. Thus $\chi(\Sigma)=\Sigma$ and $\chi$ preserves
the orders $\alpha$.

Since $\phi \circ \theta' = \theta \circ \chi$ and
$D\phi \cdot \omega = \omega \cdot \phi$, we have
$$
g\, D\phi \cdot \omega \cdot \theta'~
=~
g \cdot \left( \omega \cdot \phi \right) \cdot \theta'~
=~
g \cdot \omega \cdot \left( \phi \circ \theta' \right)~
=~
g \cdot \omega \cdot \left( \theta \circ \chi \right)~
=~
\left( g \cdot \omega \cdot \theta \right) \cdot \chi .
$$
Therefore
$$
\left( g\, D\phi \cdot \omega \cdot \theta' \right) \Homeo^+(X,\Sigma,\alpha)~
=~
\left( g \cdot \omega \cdot \theta \right) \Homeo^+(X,\Sigma,\alpha),
$$
and so $\Phi$ is well-defined.
\end{proof}

\begin{thm}
\label{thm:Phi-smooth-embedding}
The map $\Phi_\omega:Y_\omega \to \Mcal^1(X, \Sigma, \alpha)$ is a smooth embedding.
\end{thm}

\begin{proof}
The proof consists of two parts. 
We show that $\Phi_\omega$ is smooth with injective differential 
and we show that $\Phi_\omega$ is a homeomorphism onto its image.
We divide the verification of these two parts into 
Lemmas \ref{lem:Phi-diff} and \ref{lem:Phi-homeo} below.
\end{proof}

\begin{lem}
\label{lem:Phi-diff}
The map $\Phi_\omega$ is a smooth map with injective differential.
\end{lem}

\begin{proof}
By Lemma \ref{lem:free-action}, the quotient map $\Teich^1(X, \Sigma, \alpha) \to \Mcal^1(X, \Sigma, \alpha)$
is a smooth covering map. 
We define a local lift $\tPhi_\omega$ of $\Phi_\omega$ as follows. 
Let $x_0 \in \dot{X}_{\omega}$ and let $U$ be a simply connected open set 
that contains $\sigma_d$ and $x_0$ and 
that does not intersect $\omega^{-1}(0)=\{\sigma_1,\ldots, \sigma_{d-1}\}$. 
For each $x \in U$ there exists an orientation-preserving homeomorphism $\theta_x:X \to X$ 
such that $\theta_x|_{X \setminus U} ={\rm id}$ and 
$\theta_x(\sigma_d)= x$. Since $U$ is simply connected, the Alexander trick
implies that $\theta_x$ is isotopic to the identity through homeomorphisms
that are supported on $U$.
Note that $\omega \cdot \theta_x = \theta_x^*(\omega) \in \Omega(X, \Sigma, \alpha)$
and hence we 
may define $\tPhi_\omega: SL_2(\Rbb) \times U \to \Teich^1(X, \Sigma, \alpha)$
by 
\begin{equation}
\tPhi_\omega\left(g,x\right)~
:=~
(g \cdot \omega \cdot \theta_x)\, \Homeo^0(X,\Sigma).
\end{equation}
If  $p : SL_2(\Rbb) \times \dot{X}_\omega \to Y$ and
$q : \Teich^1(X,\Sigma,\alpha) \to \Mcal^1(X,\Sigma,\alpha)$ denote the quotient maps, note that $q \circ \tPhi_\omega$ equals the restriction of 
$\Phi_\omega \circ p$ to $SL_2(\Rbb) \times U$. Indeed, if 
$\theta_x$ is an orientation-preserving self-homeomorphism  of $X$  
such that $\theta_x|_{X \setminus U} ={\rm id}$ and 
$\theta_x(\sigma_d)= x$, then 
\begin{eqnarray*}
\Phi_\omega(p(g,x))
&=&
\Phi_\omega((g,x)\Aff^+_\omega) \\
&=&
(g \cdot \omega \cdot \theta_x)\, \Homeo^+(X, \Sigma, \alpha) \\
&=&
q\left( (g \cdot \omega \cdot \theta_x)\, \Homeo^0(X, \Sigma) \right) \\
&=&
q(\tPhi_\omega(g,x)).
\end{eqnarray*}
Thus, since $p$ and $q$ are smooth covering maps, it suffices 
to show that $\tPhi_\omega$ is smooth on  
$SL_2(\Rbb) \times U$ with injective differential at each 
$(g_0,x_0)$.

The smooth structure on $\Teich^1(X, \Sigma, \alpha)$ is 
given by period coordinates (see equation (\ref{eqn:period-coordinates})).
Choose a basis $\{\gamma_1, \ldots, \gamma_{r-1}\}$ for the 
relative homology group 
$H_1(X, \{\sigma_1, \ldots, \sigma_{d-1}\};\Zbb)$.
Let $\gamma_r$ be an arc that joins $\sigma_1$ to 
$\sigma_d$. Then $\{\gamma_1, \ldots, \gamma_{r}\}$
is a basis for $H_1(X, \Sigma;\Zbb)$, and 
\begin{equation}
\label{eqn:Per-defn}
\Per \left(\eta \, \Homeo^0(X, \Sigma) \right)~
:=~
\left(   
\int_{\gamma_1}  \eta,\, 
\int_{\gamma_2} \eta,\, 
\ldots,\, 
\int_{\gamma_{r}} \eta
\right)
\end{equation}
defines period coordinates in a neighborhood of 
$(g_0 \cdot \omega \cdot \theta_{x_0}) \Homeo^0(X, \Sigma)$. 
To finish the proof of the Lemma, it suffices to prove that $\Per \circ \tPhi_\omega$
is smooth on  $SL_2(\Rbb) \times U$ with injective differential
at each $(g_0,x_0)$.

The homeomorphism $\theta_x$ is isotopic to the identity
through maps that are supported on the simply connected set $U$.
If $i<d$, then $\sigma_i \in X \setminus U$ and hence $\theta_x \circ \gamma_j$
is homotopic rel endpoints to $\gamma_j$ if $j< r$. Thus, we have
\begin{eqnarray*}
\Per \circ \tPhi_\omega(g,x) 
&=&
\left(   
\int_{\gamma_1} g \cdot \omega \cdot \theta_x,\, 
\ldots,\, 
\int_{\gamma_{r-1}} g \cdot \omega \cdot \theta_x,\, 
\int_{\gamma_{r}} g \cdot \omega \cdot \theta_x
\right)
\\
&=&
\left( 
g \int_{\theta_x \circ \gamma_1} \omega ,\,
\ldots, 
g \int_{\theta_x \circ \gamma_{r-1}} \omega ,\,
g \int_{\theta_x \circ \gamma_r} \omega 
\right)
\\
&=&
 \left( 
g \int_{\gamma_1}  \omega, \,
\ldots,\, 
g\int_{\gamma_{r-1}}  \omega,\, 
 g\int_{\theta_x \circ \gamma_{r}} \omega \right).
\end{eqnarray*}
The path $\theta_{x} \circ \gamma_r$ joins $\sigma_1$ to 
$\theta_x(\sigma_d)=x$.  
In particular, the paths $\theta_{x} \circ \gamma_r$ and $\theta_{x_0} \circ \gamma_r$
may be concatenated to obtain a path that joins $x_0$ to $x$. Because $\theta_x$
is supported in $U$, this path can be homotoped to a path $\beta_x$ in $U$
that joins $x_0$ to $x$. 
In particular, 
$$
\int_{\theta_x \circ \gamma_r} \omega~
-~
\int_{\theta_{x_0} \circ \gamma_r} \omega~
=~
\int_{\beta_x} \omega.
$$ 
In sum, we can write 
$$
\Per\circ \tPhi_\omega (g,x)
~=~
\left(
g\cdot c_1,\ldots,g\cdot c_{r-1},
g\cdot\left(c_r +\nu_{x_0}(x)\right)
\right).
$$
where for $i=1, \ldots, r$ we define
$$
c_i~ :=~  \int_{\theta_{x_0}\circ \gamma_i} \omega 
\qquad 
\qquad
\nu_{x_0}(x)~ :=~ \int_{x_0}^x \omega.
$$
The form of the latter integral emphasizes the fact that the value of the 
integral does not depend on the choice of the path in $U$ that joins $x_0$ to $x$.
Since $x_0 \notin \omega^{-1}(0)$, the function $\nu_{x_0}(x)$ is smooth
with injective differential. The action of $SL_2(\Rbb)$ on $\Cbb$
defined in (\ref{eqn:g-action}) is smooth and hence 
$\Per \circ \tPhi_\omega$ is smooth. 

To prove the injectivity of the differential, 
suppose that $A$ is a trace zero $2 \times 2$ matrix and $v \in \Cbb \cong \Rbb^2$. 
Define $\alpha_{A,v}: (-\epsilon, \epsilon) \to SL_2(\Rbb) \times \dot{X}_\omega$ 
by $\alpha_{A,v}(t)= \left(g_0 e^{tA}, \nu_{x_0}^{-1}(t v) \right)$.
Then 
$$
 \Per \circ \tilde\Phi_{\omega} \circ \alpha_{A,v}(t)~
 =~
 \left(g_0 e^{t A}\cdot c_1,\ldots,  g_0 e^{t A}\cdot c_{r-1},
g_0 e^{tA} \cdot\left(c_r + tv\right)
\right).
$$
and 
so
$$
\left. \frac{d}{dt}\right|_{t=0} \Per \circ \tilde\Phi_{\omega} \circ \alpha_{A,v}(t)~
=~
\left(g_0\, A\cdot c_1,\ldots,  g_0\, A \cdot c_{r-1},
g_0\, A \cdot c_r \, +\, g_0\,  \cdot v \right)
$$

The assignment $(A,v) \mapsto \alpha_{A,v}'(0)$ identifies 
$\frak{sl}_2(\Rbb) \times \Cbb$ with $T_{g_0,x_0} (SL_2(\Rbb) \times \dot{X}_\omega)$. 
Given this identification we have 
$$
d\left(\Per \circ \tPhi_\omega\right)_{(g_0,x_0)}(A,v)
~=~
\left(
g_0\, A\cdot c_1,\, \ldots,\,  g_0\, A\cdot c_{r-1},\,
g_0\, A\cdot c_r~ +~g_0\cdot v
\right)
$$
where, as usual, $A$ is a real-linear map acting on $\Cbb$ regarded as a real vector space.
If the right hand side were zero,  then  $c_i \in \ker(g_0 A) = \ker(A)$ for each $i<r$. 
Since $X$ is compact, the absolute periods of a holomorphic 1-form span $\Cbb$ over $\Rbb$, and hence 
the vectors $c_1, \ldots, c_{r-1}$ span $\Cbb$ over $\Rbb$. 
Hence $\ker(A)=\Cbb$ and $A=0$. Thus, since the last coordinate equals zero, 
we would have $g_0\, v=0$. 
Then since $\det(g_0) >0$, we have $v=0$. 
\end{proof}

\begin{lem}
\label{lem:Phi-homeo}
The map $\Phi_\omega$ is a homeomorphism onto its image.
\end{lem}

\begin{proof}
The space $\Mcal^1(X, \Sigma, \alpha)$ is locally compact and Hausdorff, 
and hence it suffices 
to prove that the smooth map $\Phi_\omega$ is injective and proper.

We first show that $\Phi_\omega$ is injective.
Suppose $\Phi_\omega([g,x]) = \Phi_\omega([h,y])$.
Let $\theta$ (resp. $\psi$) be an orientation-preserving 
self-homeomorphism of $X$ that 
fixes $\sigma_i$ for each $i<d$ and maps $\sigma_d$ to $x$ (resp. $y$).
Then 
$$
(g \cdot \omega \cdot \theta)\, \Homeo^+(X,\Sigma, \alpha)~
=~
(h \cdot \omega \cdot \psi)\, \Homeo^+(X,\Sigma, \alpha)
$$
and hence there exists $\phi \in \Homeo^+(X,\Sigma, \alpha)$
so that $g \cdot \omega \cdot \theta = (h \cdot \omega \cdot \psi)\cdot \phi$.
It follows that 
$(h^{-1}\, g) \cdot \omega = \omega \cdot (\psi \circ \phi \circ \theta^{-1})$
and hence $\chi:=\psi \circ \phi \circ \theta^{-1}$ is an affine homeomorphism
with differential $D \chi = h^{-1}\, g$.  
Note that $\chi^{-1}(y)= \theta \circ \phi^{-1} \circ \psi^{-1}(y)= x$,
and hence $(h, y) \cdot \chi = (g,x)$. Therefore $\Phi_\omega$ is injective.

To show that $\Phi_\omega$ is proper it suffices to show that the preimage of every compact set in a compact exhaustion is compact. Towards this end, we use the relative systole (see e.g. \cite{Athreya-Masur}).
The relative systole of $\omega$ is the infimum of 
$\left|\int_\beta \omega \right|$ over all (nontrivial) saddle connections $\beta$ with endpoints in
$\{\sigma_1, \ldots, \sigma_d\}$. Here, the set of saddle connections includes 
the $\omega$-geodesic segments that join the fake zero $\sigma_d$ to itself or to another 
zero. As in the unmarked case, the relative systole descends to a continuous function 
on $\Mcal^1(X, \Sigma, \alpha)$ denoted $\sys$. 
Moreover, for each $\epsilon>0$, the  set 
$$
K_\epsilon~
:=~
\left\{
[\omega]\in \Mcal^1(X, \Sigma, \alpha)\,
:\, 
\sys( [\omega])\, \geq \epsilon
\right\}
$$
is compact and $\Mcal^1(X, \Sigma, \alpha)= \bigcup_{\epsilon >0} K_\epsilon$.

We will also use the decomposition of $G/\Gamma_\omega$ into the union of a
compact set $K'$ and finitely many `cusp neighborhoods' $N_i$. 
Each $N_i$ has the following form: There exist $k_i \in SO(2)$ and $T_i>0$ 
so that $N_i$ is the quotient map image of 
$$
\tN_i~
=~
\left\{
k\, a_t\, u_s\, k_i\,
:\,
k\in SO(2),~ t\geq T_i,\, s \in \Rbb
\right\}
$$
where 
\begin{equation}
\label{eqn:KAN-decomposition}
a_t~
=~
\begin{pmatrix}
e^{-t}&0\\
0&e^t
\end{pmatrix},
\qquad
u_s~=~
\begin{pmatrix}
1&s\\
0&1
\end{pmatrix}.
\end{equation}
Moreover, if $T_i$ is chosen sufficiently large and $h, h' \in \tN_i$,
then $h^{-1}h' \in \Gamma_\omega$ iff $h^{-1}h' = \delta_i k_i^{-1}u_s k_i$ for 
some $s \in \Rbb$
and $\delta_i \in \{ -1,1\}$.\footnote{Here the $\delta_i$ represents the ambiguity of working with $SL_2(\Rbb)$ as opposed to $PSL_2(\Rbb)$.}
Thus, if we let $\pi_i := \delta_i\, k_i^{-1}u_{s_i} k_i$ be a generator of 
$\Gamma_\omega \cap \{\delta_i k_i^{-1}u_s k_i: s \in \Rbb\}$ with $s_i>0$, then  $N_i$ corresponds to $\tN_i/\langle \pi_i \rangle$. 

Let $\epsilon>0$ and
suppose that $\Phi_\omega([g,x]) \in K_\epsilon$.  
We will first show that $g \Gamma_\omega$ lies in a compact subset of $G/\Gamma_\omega$.
If $g \Gamma_\omega$ does not belong to the compact set $K'$, 
then $g\Gamma_\omega$ belongs to some $N_i$. 

If $g \Gamma_\omega \in N_i$, then let $\phi_i \in \Aff^+_\omega$ be such that $D\phi_i = \pi_i$.
By results of \cite{Veech89},  there exists a saddle connection $\beta_i$ for $\omega$ with endpoints in $\omega^{-1}(0)$
and a positive integer $n$ so that $\phi_i^n(\beta_i)=\beta_i$.
Let $c_i:=\int_{\beta_i} \omega$. Then $\pi_i^n\, \cdot c_i = c_i$ where here $G$ acts on $\Cbb$
as in (\ref{eqn:g-action}). Since $u_{s_i} \cdot 1= 1$, we have 
$\pi_i \cdot (k_i^{-1} \cdot 1) = \delta_i\, k_i^{-1} \cdot 1$
and hence $c_i=\lambda_i\, k_i^{-1}\cdot 1$ for some nonzero real number $\lambda_i$.

Since $g\Gamma_\omega\in N_i$, there exists $\gamma_g \in \Gamma_\omega$ so that
$
g':=g\gamma_g = k\,a_t\,u_s\,k_i
$
with $k\in SO(2)$, $t\geq T_i$, and $0\leq s\leq s_i$.
Since $u_s\cdot 1=1$ and $k$ preserves the norm on $\Cbb$, we have
$$
|g'\cdot c_i|~
=~
|\lambda_i|\,
|(k\,a_t\,u_s)\cdot 1|~
=~
|\lambda_i|\,e^{-t}.
$$
Let $\phi_g \in \Aff_\omega^+$ be such that $D\phi_g = \gamma_g$.
Then $\beta_i':= \phi_g(\beta_i)$ is a saddle connection for 
$\omega$ and hence is a saddle connection for 
each $g \cdot \omega$ for each $g \in G$. Moreover,
$$
 \int_{\beta_i'} g \cdot \omega~
 =~
 g \int_{\phi_g(\beta_i)} \omega~
 =~
 g \int_{\beta_i} \gamma_g \cdot  \omega 
 =~
  g' \int_{\beta_i}   \omega~
 =~
 g' \cdot c_i.
$$
Thus, the $g \cdot \omega$-length of $\beta_i'$ equals $|g' \cdot c_i|$.

Because the endpoints of $\beta_i'$ lie in $\omega^{-1}(0)=\{\sigma_1, \ldots, \sigma_{d-1}\}$, 
there exists an orientation-preserving 
self-homeomorphism $\theta:X \to X$ that
fixes $\sigma_j$ for each $j<d$, maps $\sigma_d$ to $x$, and maps 
$\beta_i'$ to $\beta_i'$. In particular the curve $\beta_i'$ 
is a saddle connection for $g \cdot \omega \cdot \theta$ 
whose $g \cdot \omega \cdot \theta$-length equals $|g' \cdot c_i|$.
Therefore since $[g \cdot \omega \cdot \theta] =\Phi_{\omega}([g,x]) \in K_\epsilon$, we have
$$
\epsilon~
\leq~
\sys \left([g \cdot \omega \cdot \theta]\right)~
\leq~
|g'\cdot c_i|~
=~
|\lambda_i|\,e^{-t}.
$$
It follows that
$
t\leq \log(|\lambda_i|/\epsilon).
$
Thus the coset representative $g'=g\gamma_g$ lies in the compact set defined by
$$
\tK_{\epsilon,i}~
:=~
\left\{
k\,a_t\,u_s\,k_i
:\,
k\in SO(2),~
T_i\leq t\leq \log(|\lambda_i|/\epsilon),~
0\leq s\leq s_i
\right\}.
$$
In sum, since there are finitely many $N_i$ and $K'$ is compact, one finds that
there exists $A_\epsilon>0$ so that 
if $\Phi_\omega([g, x]) \in K_\epsilon$,
then there exists $\gamma_g \in \Gamma_\omega$ such that $\|g \gamma_g\|\leq A_\epsilon$ where here $\|\cdot\|$ denotes the operator norm.

Choose $\phi_g \in \Aff^+_\omega$ so that $D\phi_g =\gamma_g$. 
As above, set $g' = g \gamma_g$ and also set $x'=\phi_g^{-1}(x)$.
Then by the definition of $Y_\omega$, we have
$
[g,x]
=
\left[g', x' \right].
$
Thus $\Phi_\omega([g', x'])$ belongs to $K_\epsilon$, and we have
$$
\epsilon~
\leq~
d_{g'\cdot\omega}\left(x',\omega^{-1}(0)\right).
$$
On the other hand,
$$
d_{g'\cdot\omega}\left(x',\omega^{-1}(0)\right)~
\leq~
\|g'\|\,
d_\omega\left(x',\omega^{-1}(0)\right)~
\leq~
A_\epsilon\,
d_\omega\left(x',\omega^{-1}(0)\right).
$$
Combining inequalities gives 
$$
d_\omega\left(x',\omega^{-1}(0)\right)~
\geq~
\frac{\epsilon}{A_\epsilon}.
$$
This condition on $x'$ defines a compact subset $\oK_\epsilon \subset \dot{X}_\omega$. 

In sum, we have shown that if $\Phi_\omega([g, x]) \in K_\epsilon$, then there exists 
$\phi_g \in \Aff^+_\omega$ and $i$ so that $(g, x)\cdot \phi_g \in \tK_{\epsilon,i} \times  \oK_\epsilon$.
It follows that $\Phi_\omega^{-1}(K_\epsilon)$ is a compact subset of $Y_\omega$. Hence $\Phi_\omega$ is
proper.
\end{proof}

\begin{prop}
\label{prop:ergodic}
The measure $(\Phi_\omega)_*(\nu)$ is ergodic with respect to the 
$SL_2(\Rbb)$-action on $\Mcal^1(X,\Sigma, \alpha)$. 
Hence $\nu$ is ergodic with respect to $SL_2(\Rbb)$-action on $Y_\omega$.
\end{prop}

\begin{proof}
Because $\nu$ is $G$-invariant and $\Phi$ intertwines 
the $G$-actions, the measure $(\Phi_\omega)_*(\nu)$ is $G$-invariant.
Hence $(\Phi_\omega)_*(\nu)$ has a decomposition into ergodic measures. 
By Theorem 1.4 in \cite{EM} each ergodic 
component is an affine measure  on an affine invariant
submanifold  of the manifold $\Mcal^1(X, \Sigma, \alpha)$.\footnote{See the last remark of \S \ref{sec:AG} where it is explained why $G$-ergodicity implies $P$-ergodicity.}
By Proposition 1.3 in \cite{EM},
there are at most countably many manifolds that support such measures.
Two ergodic affine measures supported on the same affine manifold are positive multiples of one another \cite{EM}. Hence we have only
countably many such measures $\mu_i$ in the decomposition.\footnote{Note that at the end of the introduction to \cite{EM}, the authors 
state that all results in the introduction including Proposition 1.3 and 
Theorem 1.4 are valid for marked points.}

The support $N_i$ of $\mu_i$ lies in the support of 
$(\Phi_\omega)_*(\nu)$, and hence each $N_i \subset \Phi(Y_\omega)$.
By  Theorem \ref{thm:Phi-smooth-embedding}  the latter is a smooth closed 
submanifold of $\Mcal(X, \Sigma, \alpha)$.  
Since $\nu$ lies in the Lebesgue measure class 
of $Y_\omega$, the measure $(\Phi_\omega)_*(\nu)$ lies in the
Lebesgue measure class of the manifold $\Phi(Y_\omega)$.
Thus if $\dim(N_i) < \dim(\Phi(Y_\omega))$ then 
$(\Phi_\omega)_*(\nu)(N_i) =0$. If 
$\dim(N_i) = \dim(\Phi(Y_\omega))$, then since $N_i$
is closed and $\Phi(Y_\omega)$ is connected we have 
$N_i = \Phi(Y_\omega)$.
Therefore, $(\Phi_\omega)_*(\nu)$ is a sum of ergodic affine measures supported on 
the affine invariant manifold $\Phi(Y_\omega)$. But two ergodic affine measures 
with support equal to the same affine manifold are constant multiples of one another, and so $(\Phi_\omega)_*(\nu)$ is an ergodic affine measure.
\end{proof}

Recall from (\ref{eqn:pi-omega}) the representation $\pi_{\omega}$ 
of $SL_2(\Rbb)$ on $L^2(Y_\omega, \nu)$ defined by
$$
(\pi_{\omega}(g) \cdot u)(y)~
=~
u(g^{-1} \cdot y).
$$
Let $\sigma_{\omega}$ be the supremum of $\Re(s) \in [0,1)$ 
such that $\pi_s$ appears in the decomposition of $\pi_\omega$.

\begin{thm}
\label{thm:pi-omega-gap}
The representation $\pi_\omega$ has a spectral gap. That is, we have $\sigma_{\omega} < 1$.
\end{thm}
\begin{proof} 
The map $\Phi_\omega$ intertwines the representation $\pi_\omega$ and the representation 
$\pi_{\mu}$ in Theorem \ref{thm:AG} where $\mu= (\Phi_\omega)_*(\nu)$. Since $(\Phi_\omega)_*(\nu)$
is ergodic, we obtain the claim.
\end{proof}

\vspace{.5cm}

\section{Shrinking targets and Diophantine approximation}\label{sec:shrinking}

\ \medskip

In this section, we use the 
spectral gap in Theorem \ref{thm:pi-omega-gap} 
and the quantitative mean ergodic theorem for $SL_2(\Rbb)$ to obtain a decay estimate for the ergodic error 
function $\epsilon(t)$ of \S \ref{sec:duality} in the context
of the bundle $Y_\omega= (SL_2(\Rbb) \times X^*_{\omega})/\Gamma_\omega$. 
We then use the results of \S \ref{sec:duality} to 
derive consequences for the shrinking target problem
associated to the $\Aff^+_\omega$-action on $X$.

We first choose a function $\rho$ that measures the `size' of elements in 
$SL_2(\Rbb)$:  Define $\rho: SL_2(\Rbb) \to [0, \infty)$ by $\rho(g)= \frac12 \log( \tr(g^t\cdot g))$. It is well known that
$\rho(g \cdot h) \leq \rho(g) + \rho(h)$ for each $g,h \in SL_2(\Rbb)$.
Also note that for $g \in SL_2(\Rbb)$ we have $\rho(g^{-1})= \rho(g)$.
Hence the results 
of \S \ref{sec:duality} apply with this $\rho$  where $C_{\rho}=0$. For each 
$t \in \Rbb^+$ we define $G_t:= \rho^{-1}[0,t]$.

For the convenience of the reader, we recall that $\epsilon(t)$
equals the supremum of the $L^2(Y_\omega, \nu)$-norm of 
\begin{equation}
\label{eqn:mean-ergodic2}
\frac{1}{m(G_t)} \int_{G_t}
u(g^{-1}\cdot y)\, dm(g)~
-~
\frac{1}{\nu(Y_\omega)}\int_{Y_\omega} u\, d \nu
\end{equation}
taken over all unit vectors $u \in L^2(Y_\omega, \nu)$.

\begin{lemma}
\label{lem:MET-est}
Suppose that $\Gamma_\omega$ is a lattice in $SL_2(\Rbb)$.
There exists $C$ so that if $t>\log \sqrt{2}$ then
$$
\epsilon(t)~
\leq~
C\,
t
\cdot 
e^{(\sigma_\omega -1)t}.
$$
\end{lemma}

\begin{proof}
It suffices to apply the quantitative mean ergodic theorem for $SL_2(\Rbb)$
that is described in  Appendix \ref{sec:MET}. 
The action of $G$ on $(Y_\omega, \nu)$ is ergodic by Proposition 
\ref{prop:ergodic}.
The set $B_r$ that is defined in formula (\ref{eqn:B_r}) 
in Appendix \ref{sec:MET} equals the set $G_t$ iff 
$t = \frac12 \log(r^2 + r^{-2})$. Indeed, using the cyclicity of the trace we have 
$$
\rho(k_1\, g_a\, k_2)~
=~
\frac12\,
\log
\tr 
\left(
k_2^t\, g_a^t\, k_1^t\, k_1\, g_a\, k_2
\right)~
=~
\frac12\, \log \tr \left(g_a^2\right)~
=~ 
\frac12\, \log \left( a^2 + a^{-2} \right).
$$
Note that 
$e^{2t}= r^2 +r^{-2} = r^2(1+ r^{-4})$ and so since 
$1 \leq 1 +r^{-4} \leq 2$, we find that $r \leq e^t \leq \sqrt{2} r$. 
Therefore, with $t = \frac12 \log(r^2 + r^{-2})$ we have 
$$
 \log(r) \cdot r^{\sigma_\omega-1}~
 \leq~
 t\cdot 2^{-\frac12 (\sigma_\omega -1)}\, e^{(\sigma_\omega-1)t},
$$
and hence the claim follows from 
Theorem \ref{thm:MET}.
\end{proof}

By combining Lemma \ref{lem:MET-est} with Proposition \ref{prop:abstract-target}, we prove the following.

\begin{prop}
\label{prop:applied-target-X}
Suppose that $\Gamma_\omega$ is a lattice in $SL_2(\Rbb)$. 
Let $T_n \subset X$ be a sequence of measurable subsets and let $t_n$ 
be a  sequence of positive real numbers that limit to $\infty$. If 
\begin{equation}
\label{eqn:series}
\sum_{n \geq 1} \frac{t_n^2  \cdot e^{2(\sigma_\omega-1)t_n} }{\mu_{\omega}(T_n)}
\end{equation}
is finite,
then there exists $C$ so that for almost every $x \in X$ there exists $n_x$ 
such that if $n \geq n_x$, then 
there exists $\phi \in \Aff_\omega^+$ so that
\begin{enumerate}

\item  $\phi (x) \in  T_n$, and 

\item   $\rho(D\phi)\,  \leq \, t_n + C$.
\end{enumerate}
\end{prop}

\begin{proof} 
Assume that the series in (\ref{eqn:series}) 
converges. Let $\pi: \dot{X}_{\omega} \to X^*_{\omega}$ be the natural projection. Let $\dot T_n = T_n \setminus \omega^{-1}(0)$ and note that $\mu_{\omega}(\dot T_n) = \mu_{\omega}(T_n).$ By the definition of $\mu^*$ and since 
$\dot T_n \subset \pi^{-1}(\pi(\dot T_n))$, we have 
$$
\mu^*(\pi(\dot T_n))~
=~
\mu_{\omega}(\pi^{-1}(\pi(\dot T_n)))~
\geq~
\mu_{\omega}(\dot T_n)~
=~
\mu_{\omega}(T_n).
$$
Therefore, since the series in (\ref{eqn:series}) converges,
the series
\begin{equation*}
\sum_{n \geq 1} \dfrac{t_n^2 \cdot e^{2(\sigma_{\omega} - 1)\cdot t_n}}{\mu^*(\pi(\dot T_n))}
\end{equation*} 
also converges. 
But by Lemma \ref{lem:MET-est}, we have 
\begin{equation*}
\sum_{n \geq 1} \frac{\epsilon(t_n)^2}{ \mu^*(\pi(\dot T_n))}~
\leq~
C^2
\sum_{n \geq 1} \dfrac{t_n^2 \cdot e^{2(\sigma_{\omega} - 1)\cdot t_n}}{\mu^*(\pi(\dot T_n))}
\end{equation*} 

and so the former series also converges. 
Thus Proposition \ref{prop:abstract-target}
implies that
there exists a constant $C$ and a subset $E \subset X^*_{\omega}$ with $\mu^*(E) =\mu^*(X^*_{\omega})$ so that for each $[x] \in E$, there exists $n_{[x]} \in \Nbb$ 
such that if $n \geq n_{[x]},$ then there exists $\gamma_n \in \Gamma_{\omega}$ so that $\gamma_n \cdot [x] \in \pi(\dot T_n),$ and $\rho(\gamma_n)\leq t_n + C.$

Because $\pi^{-1}(E)$ has full measure in $\dot{X}_{\omega}$, to finish 
the proof of the proposition it suffices to show that for each 
$x \in \pi^{-1}(E)$ and each $n \geq n_{[x]}$, there exists 
$\phi \in \Aff_\omega^+$ such that (1) and (2) hold.

Given $x \in \pi^{-1}(E)$ and $n \geq n_{[x]}$, choose $\psi_n \in \Aff_{\omega}^+$ so that $D(\psi_n) = \gamma_n.$ 
Then $\gamma_n \cdot [x] = [\psi_n (x)] \in \pi(\dot T_n)$ and we conclude that $\psi_n(x) \in \pi^{-1}(\pi(\dot T_n)).$ 
Since $\pi$ is a normal covering with deck group $\ker(D)$, there exists $\tau_n \in \ker(D)$ such that $\tau_n \circ \psi_n (x) \in \dot T_n \subset T_n.$ 
Thus, if we define $\phi_n := \tau_n \circ \psi_n$, then
$(1)$ holds. 
Moreover, by the chain rule, $D\phi_n = D(\tau_n\circ \psi_n) = D(\psi_n) = \gamma_n,$ and so $\rho(D\phi_n) =\rho(\gamma_n) \leq t_n + C$ and $(2)$ holds.
\end{proof}

\begin{eg}
\label{eg:applied-target}
\
Suppose $t_n=n$ and $\mu_\omega(T_n) \geq \delta \cdot 
n^4 \cdot e^{2(\sigma_\omega-1)n}$ for some $\delta >0$.
Then Proposition \ref{prop:applied-target-X} implies that
there exists $C_\delta$ so that for a.e. $x \in X$ and all sufficiently 
large $n$ there exists $\phi\in \Aff_\omega^+$ 
with $\rho(D \phi) \leq n+ C_\delta$ so that $\phi(x) \in T_n$.
\end{eg}

Define $\|g\|:= \sqrt{\tr(g^tg)}$ and note that $\rho(g)= \log(\|g\|)$.
Let $d_\omega$ denote the distance function on $X$
associated to the arclength defined by integrating 
$|\omega|$. Recall the number $\beta_\omega$ defined in the introduction
and note that $\beta_\omega= 1 -\sigma_\omega$.

\begin{restatable}{thm}{thmA}[compare Theorem \ref{thm:full-measure} and Theorem \ref{thm:alpha-beta}]
\label{thm:diophantine}
Suppose that $D\Aff^+_\omega$ is a lattice in $SL_2(\Rbb)$. For each $y \in X$ and for a.e. $x \in X$, the set 
$$
H_{x,y,\alpha}~
:=~
\left\{
\phi \in \Aff_{\omega}^+~
:~ 
d_{\omega}(\phi (x), y)~
\leq~
\| D \phi \|^{-\alpha} \right\}
$$
is finite when $\alpha > 1$ and is infinite when $\alpha < \beta_{\omega}.$ 
\end{restatable}

Note that Theorem \ref{thm:full-measure} together with Theorem 
\ref{thm:alpha-beta} is weaker than Theorem \ref{thm:diophantine}
in the sense that the constant $\alpha_\omega$ in the conclusion of Theorem \ref{thm:full-measure} concerns only a.e. $(x,y) \in X \times X$.

\begin{proof}
To prove that $H_{x,y,\alpha}$ is finite when $\alpha>1$, we will apply 
the `easy' direction of the Borel-Cantelli lemma.
Let $T_\phi$ be the ball of radius $\|D\phi\|^{-\alpha}$ centered
at $y$, and define $E_\phi= \{x \in X: \phi( x) \in T_\phi\}=
\phi^{-1} (T_\phi)$.
Because $\phi^{-1}$ is measure preserving, there exists a positive integer $M$ such that we have
$
  \mu_{\omega} \left( E_\phi \right)
  =
  \mu_{\omega} \left(  T_\phi \right)
  \leq
  M \pi\, \|D \phi\|^{-2 \alpha}.
$
Note that the positive integer $M$ depends on the stratum of $\omega$ and allows for $y \in \omega^{-1}(0).$ Therefore if we can prove that 
\begin{equation}
\label{eqn:easy-BC}
\sum_{\phi \in \Aff^+_\omega} \|D \phi\|^{-2 \alpha}~
<~
\infty
\end{equation}
then by the Borel-Cantelli Lemma 
we have 
$$
\mu_{\omega} 
\left( 
\bigcap_{n=1}^\infty \bigcup_{\rho(D\phi) \geq n} E_\phi
\right)~
=~
0,
$$
and so the set of $x$ lying in 
infinitely many $E_\phi$ has measure equal to zero.

We next reduce (\ref{eqn:easy-BC}) to the statement that 
$\sum_{\gamma \in \Gamma_\omega} \|\gamma \|^{-2 \alpha} < \infty$  
where $\Gamma_\omega=D \Aff^+_\omega$.
Recall that $\Gamma_{\omega}$ is the image under $D$ of the subgroup 
$\Aff_{\omega}^+.$ Since $\ker(D)$ is finite, for any $\gamma \in \Gamma_{\omega}$ we have that there are $|\ker(D)|$ elements in $\Aff_{\omega}^+$ such that $D\phi = \gamma$ for each $\phi$ in this set. Therefore, 
\begin{eqnarray*} 
\sum_{\phi \in \Aff^+_{\omega}} \|D\phi\|^{-2\alpha}
&=&
\sum_{\gamma \in \Gamma_{\omega}} \| \gamma \|^{-2\alpha} \cdot |\ker(D)|\\
&=&
|\ker(D)| \cdot
\sum_{\gamma \in \Gamma_{\omega}} \| \gamma \|^{-2\alpha}
\end{eqnarray*}

We make three observations. Let $\hat{\Gamma}_n := \Gamma_{\omega} \cap( G_n\setminus G_{n-1}).$ First, we have the following effective lattice point counting asymptotics: there exists an $\eta >0$ such that if  $\varepsilon \in (0, \eta)$, then 
$$
\vert \hat{\Gamma}_n \vert~
=~
C_1\, m(G_n\setminus G_{n-1}) (1 + O(e^{-(\eta - \varepsilon)n}))
$$
where $C_1$ is a constant.
See, for example, \cite{Gorodnik-Nevo-10-lattice} and references therein. Second, we can bound the size of $\| \gamma \|$ for $\gamma \in \hat{\Gamma}_n$. Indeed, for any such $\gamma$, 
\begin{align*}
n-1 &< \log \| \gamma \| \leq n \\
e^{n-1} &< \| \gamma \| \leq e^n.
\end{align*}
Third, by integrating the characteristic function of $G_t$ over $SL_2(\mathbb{R})$ with respect to the Haar measure $m$, for some constant $C_2$, 
we have $m(G_t) = C_2(e^{2t} - 2)$ for $t \geq \log \sqrt{2}$.

Applying these observations to the sum, and absorbing constants as needed into $C_3$, we observe that the sum converges when $\alpha > 1:$

\begin{align*}
\sum_{n=2}^{\infty} \sum_{\gamma \in \hat{\Gamma}_n} \| \gamma \|^{-2\alpha} & \leq C_3 \sum_{n=2}^{\infty} m(G_n \setminus G_{n-1}) e^{-2\alpha(n-1)} \\
&\leq C_3 \sum_{n=2}^{\infty} (e^{2n}-e^{2(n-1)})e^{-2\alpha(n-1)} \\
&= C_3 \sum_{n=2}^{\infty} e^{2(n-1)}(e^2 - 1) e^{-2\alpha(n-1)} \\
&= C_3 \sum_{n=1}^{\infty} e^{2n} e^{-2\alpha n}.
\end{align*}

Finally, we show that if $\alpha < \beta_{\omega} = 1 - \sigma_\omega$,
then the set $H_{x,y,\alpha}$ is infinite for a.e. $x$. 
Let $t_n = n$ and let $T_n$ be the disc of radius 
$n^2 \cdot e^{(\sigma_\omega-1)n}$ centered at $y$. 
Note that there exists $r_0$ so that if $0<r< r_0$
and $y \in X$, then $\mu_\omega(B_r(y)) \geq \pi r^2$.
Therefore, for all sufficiently large $n$ we have 
$\mu_{\omega}(T_n) \geq \pi n^4 \cdot e^{2(\sigma_\omega-1)n}$, and 
we are in  the situation of 
Example \ref{eg:applied-target}. Thus, there exists $C$ so that 
for a.e. $x$ there exists $n_x$ such that if $n \geq n_x$,
then there exists $\phi_n \in \Aff^+_\omega$ with 
\begin{equation}
\label{eqn:example-construction}
d_\omega(\phi_n(x),y)~ <~ n^2 \cdot (e^{n})^{\sigma_\omega-1}
\end{equation}
and $\log(\|D\phi_n\|) \leq n +C$.  The latter implies that 
$\|D\phi_n \|^{-\alpha}\geq (e^{n})^{-\alpha} \cdot e^{-C \alpha}$. 
Since $\alpha < 1 - \sigma_\omega$, for sufficiently large $n$ we have 
$n^2 e^{-(1-\sigma_\omega - \alpha)n} \leq e^{-C\alpha}$ and thus $n^2 e^{-(1-\sigma_\omega)n} \leq e^{-C\alpha}e^{-n\alpha}$.
By combining these estimates, we find that for all sufficiently large $n$ 
\begin{equation}
\label{eqn:dio-est}
   d_{\omega}(\phi_n (x), y)~
\leq~
\| D \phi_n \|^{-\alpha}.
\end{equation}
In particular, each $\phi_n$ lies in $H_{x,y,\alpha}$.
To finish the proof, it suffices to
show that the map $n \mapsto \phi_n$ 
has infinitely many values.

The orbit $\Aff^+_{\omega} \cdot y$ has measure zero, and 
hence we may suppose, without loss of generality, that $x \notin \Aff^+_{\omega} \cdot y$. Estimate (\ref{eqn:example-construction}) implies 
that $\phi_n(x)$ converges to $y$, and so if $n \mapsto \phi_n$ had 
finitely many values, then we would have $\phi_n(x) =y$ for infinitely 
many $n$. Then $x$ would lie in the $\Aff^+_\omega$ orbit of $y$,
contradicting our assumption. Therefore, the map
$n \mapsto \phi_n$ has infinitely many values, and thus 
$H_{x,y, \alpha}$ is infinite.
\end{proof}

\ \appendix

%%%%%%%%%%%%%%%%%%%%%%%%%%%%%%%%%%%%%%%%%%%%%%%%%%%%%%%

\ \medskip

\section{The quantitative mean ergodic theorem for $SL_2(\Rbb)$}
\label{sec:MET}

\ \medskip

Let $G:=SL_2(\Rbb)$ and let $K:=SO(2)$.
Let $(Y, \nu)$ be a Borel measure space with $\nu(Y) < \infty$. 
Suppose that  $G$ acts ergodically on the left of 
$(Y,\nu)$ by measure preserving transformations.
Define a unitary representation $\pi$ of $G$ on $L^2(Y, d\nu)$
by setting
\begin{equation}
\label{eqn:left-regular}
(\pi(g) u)(y)~
:=~
u(g^{-1}\cdot y).
\end{equation}
Define
$$
g_a~ 
=~
\left(
\begin{array}{cc}
a & 0 \\
0 & a^{-1}
\end{array}
\right).
$$
Recall the polar decomposition: Each $g \in G$ can be written 
as $g= k_1\, g_a\, k_2$ where $a \geq 1$ and $k_1, k_2 \in K$.
For each $r \geq 1$, define 
\begin{equation}
\label{eqn:B_r}
B_r~ 
:=~
\left\{ 
k_1\, g_a\, k_2\, :\, k_1,k_2 \in K \mbox{ and }  1 \leq a \leq r
\right\}.
\end{equation}

Recall from \S \ref{sec:AG} the description of the irreducible spherical representations
$\pi_s$ of $G$.
The following is well-known. See, for example, \cite{Nevo} and \cite{Ananetal}.

\begin{thm}
\label{thm:MET}
Let $\sigma$ be the supremum of the ${\rm Re}(s)$ such that the 
nontrivial representation $\pi_s$ appears in the decomposition of $\pi$.
Fix $R > 1.$ There exists $C>0$ such that for every $r > R$, 
\begin{equation*}
\label{eqn:mean-ergodic-state}
\left\| 
\frac{1}{m(B_r)} \int_{B_r} u(g^{-1} \cdot y)\,  dm(g)~
-~
\frac{1}{\nu(Y)}\int_{Y} u(y)\, d \nu(y)
\right\|~
\leq~
C\,
\|u\|\,
\log(r)
\cdot 
r^{\sigma-1}.
\end{equation*}
\end{thm}

\ \medskip

\ \medskip

%%%%%%%%%%%%%%%%%%%%%%%%%%%%%%%%%%

\section{The theorem of Avila-Gou\"ezel in the case of marked points}
\label{app:avila-gouezel}

\ \medskip

The purpose of this appendix is to describe how to modify the methods of \cite{Avila-Gouezel-13-small-eigenvalues} to
obtain Theorem \ref{thm:AG}.

\thmAG*

The methods of
Avila and Gou\"ezel give Theorem \ref{thm:AG} in the case where
each cone angle $2\pi \kappa_i = 2 \pi(\alpha_i+1)$ is at least $4 \pi$.
Our goal is to show that the methods work equally well when some of the cone angles 
equal $2\pi$, that is when we have fake zeros.\footnote{
The authors of 
\cite{Avila-Gouezel-13-small-eigenvalues} make no explicit restriction
on the cone angle $2 \pi \kappa_i = 2\pi(\alpha_i+1)$ 
 other than what is required
by Gauss-Bonnet: $\sum_{i = 1}^d (\kappa_i - 1) = 2g-2.$
}

The dependence of \cite{Avila-Gouezel-13-small-eigenvalues} on the 
nature of the zeros is confined to three components: 
\begin{itemize}

\item the local product structure of the measure $\mu$ shown in Proposition 4.1 \cite{Avila-Gouezel-13-small-eigenvalues},

\item properties of the `Finsler' norm of \cite{AGY-06-exp-mix}, and

\item  the recurrence estimate of Proposition 6.1 \cite{Avila-Gouezel-13-small-eigenvalues}. 
\end{itemize}
The remaining components of the proof in \cite{Avila-Gouezel-13-small-eigenvalues}---the reduction argument in \S 3 and Appendix A, the application of the `Finsler' norm in \S 5, the construction of the Banach space in \S 7, and the proof of the main spectral result in \S 8 --- involve the nature of the zeros only through the information that 
they receive from the three components listed above. In what follows, 
we explain why the results that constitute each of these 
three components hold in the setting of (finitely many) fake zeros.

\medskip
%%%%%%%%%%

\subsection{Proposition 4.1 in \cite{Avila-Gouezel-13-small-eigenvalues}}
\label{subsec:Prop-4.1-AG13}
\ \\
Let $\tmu_1$ be the lift of the finite ergodic $SL_2(\Rbb)$-invariant 
measure $\mu$ to $\Teich^1(X, \Sigma, \alpha)$, and let $\Ncal$ 
be the submanifold on which it is supported.  
The measure $\tmu_1$ 
is invariant under the $\Rbb$-action on $\Teich^1(X, \Sigma, \alpha)$ 
defined by $g_t(x):= \begin{bmatrix} e^t & 0 \\ 0 & e^{-t} \end{bmatrix} \cdot x$. 
Integration of a 1-form defines an element of $\Hom(H_1(X,\Sigma;\Zbb), \Cbb)$ and thus a map 
$\Phi: \Teich(X, \Sigma, \alpha) \to H^1(X,\Sigma; \Cbb)$. The map 
$\Phi$ is a local diffeomorphism (see, e.g. \cite{Veech90}).
By evaluating each cohomology class on a basis $\gamma$ 
for $H_1(X,\Sigma;\Zbb)$ we obtain a complex linear 
isomorphism $\alpha: H^1(X,\Sigma; \Cbb) \to \Cbb^r$ and 
the map $\Per$ defined in (\ref{eqn:Per-defn}) equals $\alpha \circ \Phi$.
Hence a restriction of $\Phi$ to an open set on which it is injective
is called `period coordinates'. In particular, $\Phi$ determines 
the complex affine structure on $\Teich(X, \Sigma, \alpha)$.

The map $(x,v) \mapsto (x, D\Phi_x(v))$  
identifies the tangent bundle $T \Teich(X, \Sigma, \alpha)$ with the trivial bundle 
$\Teich(X, \Sigma, \alpha) \times  H^1(X,\Sigma;\Cbb)$. The map $(Dg_t)_x: T_x \Teich(X, \Sigma, \alpha) \to
T_{g_t(x)} \Teich(X, \Sigma, \alpha) $ corresponds under this identification to multiplication by $\begin{bmatrix} e^t & 0 \\ 0 & e^{-t} \end{bmatrix}$ on $\Cbb$-valued cohomology classes induced by the action in
(\ref{eqn:g-action}).  In particular, if $x \in  \Teich(X, \Sigma, \alpha)$ and 
$v, w\in H^1(X, \Sigma; \Rbb)$ then $Dg_t$ may be viewed as the 
map $(x,v  + iw) \mapsto (g_t x, e^t v + i e^{-t}w)$ where we have used the natural 
identification of $H^1(X,\Sigma;\Cbb)$ with 
$H^1(X,\Sigma; \Rbb)\otimes_\Rbb \Cbb$. The flow $g_t$ preserves $\Ncal$ 
and $Dg_t$ maps $T \Ncal$ to itself. Let $E^u$ denote the subbundle of $T \Ncal$  with fibers 
$T_x \Ncal \cap D\Phi_x^{-1}(H^1(X,\Sigma;\Rbb) \otimes_\Rbb \Rbb)$
and let $E^s$ denote the subbundle with fibers $T_x \Ncal \cap 
D\Phi_x^{-1}(H^1(X,\Sigma;\Rbb) \otimes_\Rbb i\Rbb)$.

By Theorem 1.4 in \cite{EM}, the measure $\mu$ is affine.\footnote{See the last remark of \S \ref{sec:AG}.}  Let $\Ccal := GL_2^+(\Rbb)\cdot \Ncal = \Rbb_{>0} \cdot \Ncal.$ For each $x \in \Ncal,$ there exists an open neighborhood $U \subset \Ncal$ containing $x$ such that $\Phi(\Rbb_{>0} \cdot U)$ is an open set in a complex-linear subspace defined over $\Rbb$ in period coordinates. Call this space $V$ and we have $V = V_{\Rbb}  \oplus i V_{\Rbb}$ for $V_{\Rbb} = V \cap H^1(X,\Sigma; \Rbb).$ Let $k = \dim_{\Rbb} V_{\Rbb}.$ 
Let $\mathrm{Area}([\omega]) = \frac{i}{2} \int_X \omega \wedge \oomega$ and define $P: \Rbb_{>0} \times \Ncal \to \Ccal$ by $P(a,x) = \sqrt{a} x.$ Then $\mathrm{Area}(P(a,x)) = a$ and since $V$ is a real 2k-dimensional vector space, a Lebesgue measure in the coordinates $P(a,x)$ is, up to normalization, $a^{k-1} da\otimes\tmu_1.$ Let $$\tmu = P_*(a^{k-1}da \otimes \tilde{\mu_1}).$$ Then $\tmu$
is an affine linear measure, i.e. it is, up to normalization, the Lebesgue measure induced by a volume form on $V.$

Let $\nu$ denote a volume form on $V_{\Rbb}$ normalized so that the volume form $\Lambda = \nu \circ \mathrm{Re} \wedge \nu \circ \mathrm{Im}$ gives $| \Lambda | = \Phi_*(\tmu),$ where $| \Lambda |$ denotes the induced Lebesgue measure. Let $\langle \cdot , \cdot \rangle$ denote the intersection form on $H^1(X,\Sigma; \Rbb)$ given by $$\langle c, c' \rangle := \int_X p(c) \wedge p(c')$$ 
where $p: H^1(X,\Sigma; \Rbb) \to H^1(X;\Rbb)$ is the forgetful map. For any $x \in \Ncal,$ $\langle a_x, b_x \rangle = 1.$ Indeed, for $x = [\omega],$ $\Phi([\omega]) = [\mathrm{Re}(\omega)] + i [\mathrm{Im}(\omega)],$ and $\int_X \mathrm{Re}(\omega) \wedge \mathrm{Im}(\omega) = \mathrm{Area}(x) = 1.$

The following is Proposition 4.1 \cite{Avila-Gouezel-13-small-eigenvalues} inclusive of fake zeros, applied to the measure $\tmu_1$. We use Theorem 1.4 in \cite{EM}\footnote{Note that at the end of the introduction to \cite{EM}, the authors state that all results in the introduction including Proposition 1.3 and Theorem 1.4 are valid for marked points.} and note that the results in \cite{EM} are independent of those in \cite{Avila-Gouezel-13-small-eigenvalues}.\footnote{One could prove this without appealing to \cite{EM} by validating Proposition 4.3 in \cite{Avila-Gouezel-13-small-eigenvalues} with additional fake zeros. This involves clarifying how Avila and Gou\"ezel upgrade from a pointwise bound in Lemma 2.1' \cite{Forni02} to bound that is uniform on compacta.}

\begin{prop} Let $\tmu_1$ be the lift of the finite ergodic $SL_2(\Rbb)$-invariant measure $\mu$ supported on a submanifold $\Ncal$ of $\Teich^1(X,\Sigma, \alpha).$
\begin{enumerate}
\item For every $x \in \Ncal$, there is a decomposition of the tangent space $$T_x\Ncal = \Rbb Z(x) \oplus E^u(x) \oplus E^s(x),$$ where $Z(x)$ is the direction of the $g_t$-flow and $$E^u(x) = T_x\Ncal \cap D\Phi_x^{-1}(H^1(X,\Sigma;\Rbb) \otimes_{\Rbb} \Rbb),$$ $$E^s(x) = T_x \Ncal \cap D\Phi_x^{-1}(H^1(X,\Sigma, \Rbb) \otimes_{\Rbb} i\Rbb).$$

\item The subspaces $E^s(x)$ and $E^u(x)$ depend in a $C^{\infty}$ way on $x \in \Ncal$, are integrable, and the integral leaves $W^u(x)$, $W^s(x)$ are affine submanifolds of $\Teich(X,\Sigma,\alpha).$

\item For every $x \in \Ncal$, there is a volume form $\mu_u$ on $E^u(x)$ (defined up to sign) such that $x \to \mu_u(x)$ is $C^{\infty}$. Moreover, $x \to \mu_u(x)$ is constant along the
unstable manifolds $W^u.$ Additionally, there exists a scalar $d \geq 0$ such that $(g_t)_*\mu_u = e^{-dt}\mu_u.$
\item For every $x \in \Ncal$, there is a volume form $\mu_s$ on $E^s(x)$ (defined up to sign) such that $x \to \mu_s(x)$ is $C^{\infty}.$ Moreover, $x \to \mu_s(x)$ is constant along the stable manifolds $W^s.$ Additionally, $(g_t)_*\mu_s = e^{dt}\mu_s.$
\item For every $x \in \Ncal$, the volume form $d\tmu_1(x)$ on $T_x\Ncal$ is equal to the product of $d\mathrm{Leb}$, $\mu_u(x)$, and $\mu_s(x)$ respectively in the directions $Z(x)$, $E^u(x),$ and $E^s(x).$
\end{enumerate}
All of these data are $\Mod(X,\Sigma,\alpha)$-equivariant. 
\end{prop}

\begin{proof}
For any $x\in \Ncal$ there exists an open set $U$ containing $x$ such that $\Phi|_{U}$ is a diffeomorphism. We prove each statement locally, and then provide a local-to-global argument.  

\emph{(1) } Let $x \in U \subset \Ncal$ such that $\Phi|_{U}$ is a diffeomorphism. Let $\Phi(x) = a_x + i b_x.$ Then $\langle a_x, b_x \rangle = 1,$ and since $U$ is locally cut-out by this condition,  
$$D\Phi_x\left(T_x\Ncal\right) = \left\{ v + iw \,:\, v, w \in V_\Rbb,\ 
\langle v, b_x \rangle + \langle a_x, w \rangle = 0 \right\}.$$ 
Given a non-zero vector $v + iw \in D\Phi_x\left(T_x\Ncal\right)$, let $s = \langle v, b_x \rangle$ and we have the decomposition 
$$v + iw = s\left(a_x - ib_x\right) + \left(v - sa_x\right) 
+ i\left(w + sb_x\right).$$
Each term lies in $V$. Moreover, $$\langle v - sa_x, b_x \rangle = \langle v, b_x \rangle - s\langle a_x, b_x \rangle = 0$$ and $$\langle a_x, w + sb_x \rangle =\langle a_x, w \rangle + s\langle a_x, b_x\rangle = \langle a_x, w \rangle + \langle v, b_x \rangle = 0.$$ 
Since 
\begin{eqnarray*}
D\Phi_x(E^u(x))&=& \left\{ v \in V_\Rbb \,:\, \langle v, b_x \rangle = 0 \right\}, \text{ and} \\
D\Phi_x(E^s(x)) &=& \left\{ iw \,:\, w \in V_\Rbb,\ \langle a_x, w \rangle = 0 \right\},
\end{eqnarray*}
we have that $v-sa_x \in D\Phi_x(E^u(x))$ and $i(w+sb_x) \in D\Phi_x(E^s(x)).$

$z_x := a_x - ib_x$ is the infinitesimal generator of $\begin{bmatrix} e^t & 0 \\ 0 & e^{-t} \end{bmatrix}$ and the infinitesimal generator of $g_t,$ $Z(x)= D\Phi_x ^{-1}(z_x).$ Indeed, 
\begin{equation*}
\frac{d}{dt}\bigg|_{t=0} \left(e^t a_x + ie^{-t}b_x\right) = a_x - i b_x. 
\end{equation*}  
Thus, $s(a_x - ib_x) \in \Rbb z_x.$ 

Let $c \in \Rbb$ and assume $c(a_x-ib_x) + u' + iv' = 0$ for $u' \in D\Phi_x(E^u(x))$ and $iv' \in D\Phi_x(E^s(x)).$ Then, the real part $ca_x + u' = 0$ and so $\langle ca_x + u', b_x\rangle = 0.$ Moreover, $\langle ca_x + u', b_x\rangle = c \langle a_x, b_x \rangle + \langle u',b_x \rangle = c.$ Thus, $c=0,$ and so $u' = 0$ and $v' = 0.$ We have the decomposition $D\Phi_x(T_x\Ncal) = \Rbb z_x \oplus D\Phi_x(E^u(x)) \oplus D\Phi_x(E^s(x)).$

Therefore, the tangent space splits, and since $\Phi$ is a local diffeomorphism, we have 
$$T_x\Ncal = \Rbb Z(x) \oplus E^u(x) \oplus E^s(x).$$

\emph{(2) }
The subspace $D\Phi_x(E^u(x))$ depends only on $b_x$, and hence $E^u(x)$ depends smoothly on $x$. Define $W^u(x) \subset \Ncal \cap U$ by
$$
\Phi\left(W^u(x)\right)
=
\left\{
a' + ib_x
\,:\,
a' \in V_\Rbb,\ 
\langle a',b_x\rangle = 1
\right\}
\cap
\Phi(U).
$$
Thus $\Phi\left(W^u(x)\right)$ is an open subset of an affine subspace of $V$. If $y \in W^u(x)$, then $b_y=b_x$, and hence
$$
D\Phi_y\left(T_yW^u(x)\right)
=
\left\{
v\in V_\Rbb
\,:\,
\langle v,b_y\rangle=0
\right\}
=
D\Phi_y\left(E^u(y)\right).
$$
Since $\Phi$ is a local coordinate map, it follows that
$
T_yW^u(x)=E^u(y).
$
Thus $E^u$ is integrable and its integral leaves are affine submanifolds in period coordinates. The stable case is analogous, with $a_x$ fixed.

\emph{(3) } Since $\langle a_x, b_x \rangle = 1$ for $x \in \Ncal$, $E^u(x)$ and $E^s(x)$ each have dimension $k - 1,$ where $k = \dim_{\Rbb} V_{\Rbb}$. Let $\{ e_1, \ldots, e_{k-1}\}$ be a basis for $D\Phi_x(E^u(x)).$ Define a volume form by
$$\mu_u(x)(e_1, \ldots, e_{k-1}) := \nu\left(a_x, e_1, \ldots, e_{k-1}\right).$$
This is nonzero because $a_x \notin D\Phi_x(E^u(x)).$ Moreover, $x \to \mu_u(x)$ is smooth since $\nu$ is a smooth volume form.

If $y \in W^u(x)$, then 
$\langle a_y - a_x, b_x \rangle = 0$ and $a_y - a_x \in D\Phi_x(E^u(x))$. 
Since $\nu$ is multilinear and $D\Phi_y(E^u(y)) = \left\{
v\in V_\Rbb \,:\, \langle v,b_x\rangle=0 \right\} = D\Phi_x\left(E^u(x)\right)$ as noted in (2),
we have 
$$0 = \nu(a_y - a_x, e_1, \ldots, e_{k-1}) = \mu_u(y) - \mu_u(x).$$ Thus, the form is constant along $\Phi(W^u(x)).$

Lastly, the pushforward of a volume form is the pullback of the form with respect to the inverse function: $(g_t)_*\mu_u = (g_{-t})^*\mu_u$. A computation shows
\begin{align*}
(g_{-t})^*\mu_u(x)(e_1, \ldots, e_{k-1}) &= \mu_u(g_{-t} x)(e^{-t} e_1, \ldots, e^{-t} e_{k-1}) \\
&= \nu(e^{-t} a_x, e^{-t}e_1, \ldots, e^{-t}e_{k-1}) \\
&= e^{-kt} \nu(a_{x}, e_1, \ldots, e_{k-1})\\
&= e^{-kt} \mu_u(x).
\end{align*}
Let $d = k$ and let $\mu_u$ also denote the pullback of $\mu_u$ under $\Phi.$ That completes the proof of (3).

\emph{(4) } The proof is similar to (3). Use $b_x$ in place of $a_x$ to define the form $\mu_s(x)$ on $D\Phi_x(E^s(x)).$

\emph{(5) } Let $R(x) := D\Phi(x)^{-1}(a_x + ib_x)$ be a radial vector. By construction,  $T_x\mathcal{C} = \Rbb R(x) \oplus T_x\Ncal$. $R(x)$ varies smoothly over $x \in U.$

Recall that the volume form on $V$ is given by $\Lambda = \nu\circ\mathrm{Re} \wedge \nu\circ \mathrm{Im}$ and $|\Lambda| = \Phi_*(\tmu)$ on $\Phi(\Rbb_{>0} \cdot U).$ Define a volume form on $\Phi(U)$ by $$\hat\mu = \frac12 \iota_{a_x + ib_x} \Lambda,$$ where $\iota$ denotes interior multiplication. Let $\{ e_1, \ldots, e_{k-1} \}$ be a basis in $D\Phi_x(E^u(x))$ and let $\{if_1, \ldots, if_{k-1} \}$ be a basis for $D\Phi_x(E^s(x)).$ Then, for any scalar $c,$ 
\begin{align*}
\hat\mu(c z_x, e_1, \ldots, e_{k-1}, if_1, \ldots, if_{k-1}) &= \frac{c}{2} \Lambda(a_x + ib_x, a_x-ib_x, e_1, \ldots, e_{k-1}, if_1, \ldots, if_{k-1}) \\
&= - c \Lambda(a_x, ib_x, e_1, \ldots, e_{k-1}, if_1, \ldots, if_{k-1}) \\
&=  (-1)^k c \nu(a_x, e_1, \ldots, e_{k-1}) \nu(b_x, f_1, \ldots, f_{k-1}) \\
&= (-1)^k c \mu_u(x) \mu_s(x),
\end{align*}
where in the second line we use the multilinearity of $\Lambda$ and in the third line we use that $\Lambda$ is an alternating tensor. 

Let $dt$ denote the $1$-form on $T_x\Ncal$ with $dt(Z(x))=1$ and
$dt|_{E^u(x) \oplus E^s(x)}=0$, and extend $\mu_u(x)$, $\mu_s(x)$ to $T_x \Ncal$ by composing with the projections of the splitting in (1). The computation above shows $\Phi^*\hat\mu = (-1)^k \, dt \wedge \mu_u \wedge \mu_s$ on $T_x\Ncal$.

%So, at each $x,$ $\hat\mu$ is a scalar multiple of $d\mathrm{Leb}(z_x) \wedge \mu_u(x) \wedge \mu_s(x)$ where $d\mathrm{Leb}(z_x)$ is the volume form on $\Rbb \cdot z_x$ that induces the usual Lebesgue measure.  

%Let $dt$ be the $1$-form on $T_x\Ncal$ such that $dt(Z(x)) = 1$ and $dt|_{E^u(x)\oplus E^s(x)}=0.$ Then $\Phi^*(\hat\mu) = (-1)^k dt \wedge \mu_u \wedge \mu_s.$

Moreover, the pullback of $\hat\mu$ is the volume form that induces the measure $\tmu_1$ on $U.$ The map $\Phi \circ P$ is a diffeomorphism when restricted to $\Rbb_{>0} \times U.$ Since $P(a,x) = \sqrt{a}x,$ we have $\Phi\circ P (a,x) = \Phi(\sqrt{a}x) = \sqrt{a} \Phi(x).$ By the product rule, $$D(\Phi \circ P)_{(a,x)}(s,v) = \sqrt{a} D\Phi_x(v) + \frac{s}{2\sqrt{a}}\Phi(x).$$ Note that at each point we are using the identification of the tangent space of $H^1(X,\Sigma; \Cbb)$ with itself. 

Fix $v_1, \ldots, v_{2k-1} \in T_x\Ncal.$ We have 
\begin{align*}
(\Phi \circ P)^*\Lambda \left( (0,v_1), \ldots, (0,v_{2k-1}), (1,0) \right) &= \Lambda ( D(\Phi \circ P)_{(a,x)}(0,v_1), \ldots , D(\Phi \circ P)_{(a,x)}(1,0) )\\
&= \Lambda\left(\sqrt{a} D\Phi_x(v_1), \ldots, \sqrt{a}D\Phi_x(v_{2k-1}), \frac{1}{2\sqrt{a}}(a_x + ib_x))\right)\\
&= (-1)^{2k-1} \cdot \frac12 a^{k-1} \cdot \Lambda(a_x + ib_x, D\Phi_x(v_1), \ldots, D\Phi_x(v_{2k-1})) \\
&= (-1) \cdot a^{k-1} \hat{\mu} (D\Phi_x(v_1), D\Phi_x(v_2), \ldots, D\Phi_x(v_{2k-1})).
\end{align*}

Thus, $( \Phi \circ P)^*| \Lambda | = a^{k-1}da \otimes |\hat\mu|.$ By \cite{EM} and our choice of normalization of $\nu,$ $$|\Lambda| = \Phi_*(\tmu) = (\Phi \circ  P)_*(a^{k-1} da \otimes \tilde\mu_1 ).$$ Therefore, $$( \Phi \circ P)^*| \Lambda | = a^{k-1} da \otimes \tilde\mu_1$$  and $\tilde\mu_1 = | \hat\mu|.$ We conclude that $\tilde\mu_1$ is a product of the measures induced by $\mathrm{Leb}$, $\mu_u(x),$ and $\mu_s(x)$ in the directions $Z(x),$ $E^u(x),$ and $E^s(x),$ respectively.

We next consider the global implications.
On the overlap of two charts $U, U',$ the subspaces $V$ and $V'$ in $H^1(X,\Sigma;\Cbb)$ agree. Thus, $V$ is constant along any geodesic segment $\{g_t x\}_{t=0}^{t=T}.$ The above constructions extend in each connected component. Moreover, $|\Lambda|$ is determined by $\tmu$, and so the corresponding volume forms $\nu$ and $\nu'$ agree up to sign. 

$\Mod(X,\Sigma,\alpha)$ acts by linear transformations in period coordinates preserving $H^1(X,\Sigma; \Zbb).$ The action commutes with $g_t$ and complex conjugation. Since $\Mod(X,\Sigma, \alpha)$ acts transitively on the set of connected components of the 
$\Mod(X, \Sigma, \alpha)$-saturation of $\Ncal$, the dimension of $V_{\Rbb}$ is the same in each connected component. 

Moreover, both $\langle \cdot , \cdot \rangle$ and $\tmu$ are invariant under $\mathrm{Mod}(X, \Sigma,\alpha).$  It follows that $\Rbb Z,$ $E^u$, $E^s$, $W^u$, and $W^s$ are $\mathrm{Mod}(X,\Sigma, \alpha)$-equivariant. $\mu_u$ and $\mu_s$ are $\mathrm{Mod}(X,\Sigma,\alpha)$-equivariant up to sign since $\nu$ is only determined up to sign. 
\end{proof}

\  \medskip
%%%%%%%%%%

\subsection{The `Finsler' norm of  \S 2.2.2 in \cite{AGY-06-exp-mix}} 
\label{subsec:Finsler}
\ \\
For each  $c \in H^1(X, \Sigma; \Cbb)$ and each $[\omega] \in \Teich(X, \Sigma, \alpha)$, 
Avila, Gou\"ezel, and Yoccoz define 
$$
\|c\|_{\omega}~
=~
\sup_{\gamma} \left|\frac{ c[\gamma]}{\int_\gamma \omega}\right|
$$
where the supremum is over all saddle connections $\gamma$.
Using the existence of a triangulation of $X$ by saddle connections, 
they show that $\|\cdot \|_{\omega}$ is a norm (Proposition 2.10 \cite{AGY-06-exp-mix}) and
that $\omega \to \|\cdot \|_{\omega}$ is continuous (Proposition 2.11 \cite{AGY-06-exp-mix}).
Thus, since each tangent space to $\Teich(X, \Sigma, \alpha)$ may be canonically 
identified with $H^1(X, \Sigma; \Cbb)$, they obtain a length space structure on $\Teich(X, \Sigma, \alpha)$ via integration along paths. They show that 
$[\omega] \mapsto \log(\sys([\omega]))$ is 1-Lipschitz (Lemma 2.12 \cite{AGY-06-exp-mix}) and that 
the associated metric on $\Teich(X, \Sigma, \alpha)$ is complete (Corollary 2.13 \cite{AGY-06-exp-mix}).
These facts are used in \S 5 of \cite{Avila-Gouezel-13-small-eigenvalues}.

Note that each of these facts generalizes to the context of (finitely many) fake zeros.
In the context of fake zeros, one defines a saddle connection on $(X, \omega)$
to be a geodesic segment whose endpoints lie in $\Sigma$ where $\Sigma$
now includes fake zeros as well as the `true' zeros of $\omega$. 
There is still a shortest saddle connection whose length defines a notion
of systole whose reciprocal defines a proper function on $\Mcal^1(X, \Sigma, \alpha)$.
Moreover, one may construct a triangulation\footnote{
For example, one may apply the Delaunay construction \cite{Th98}.}  
of $X$ whose 1-skeleton consists of saddle connections and hence the saddle connections generate $H_1(X, \Sigma; \Cbb)$.
With these analogues, the proofs of the generalizations of Proposition 2.10, Proposition 2.11, Lemma 2.12, and Corollary 2.13 in the context of fake cone points are exactly the same.

\  \medskip
%%%%%%%%%%

\subsection{Proposition 6.1 in \cite{Avila-Gouezel-13-small-eigenvalues}} 
\ \\

Let $\tmu_1$ and $\Ncal$ 
be the measure and submanifold introduced in \S \ref{subsec:Prop-4.1-AG13} above.
Let $\mu_u$ be the associated measure on each unstable leaf.
By restricting the `Finsler' norm to the unstable subspace of each tangent space
to $\Ncal$, we obtain a notion of arclength on each unstable leaf.
In turn, this notion of arclength induces a distance function $d_u$ on each unstable leaf.

Define $V: \Teich^1(X, \Sigma, \alpha) \to \Rbb$ by setting 
$V([\omega]) := \max \left( 1/\sys([\omega]), 1 \right)$.
In \cite{Avila-Gouezel-13-small-eigenvalues}, the authors prove

\begin{prop} \label{prop:AG-6.1}
Let $\delta \in (0, 1/4)$.  There exists $C > 0$ such that, for each
$[\omega] \in \Ncal$ and each
$t \geq 0$,
$$
\frac{1}{\mu_u( B^u_{1/100})}
\int_{B^u_{1/100}} 
V \left( g_t \cdot [\eta] \right)\, d \mu_u([\eta])~
\leq~
C\,
e^{-(1-2\delta)t} V([\omega])~ +~ C
$$
where $B^u_{1/100}$ is the disc of radius $1/100$ about $[\omega]$ with respect to the metric
$d_u$ on the leaf of $W^u$ containing $[\omega]$.
Moreover, the function $\log V$ is $1$-Lipschitz for the Finsler norm. 
\end{prop}

We will explain why Proposition \ref{prop:AG-6.1} holds when $\Sigma$ contains finitely many additional fake zeros. First note that Lemma 2.12 in \cite{AGY-06-exp-mix} holds for the adapted Finsler norm as indicated in subsection \ref{subsec:Finsler}, and hence we find that the function $\log V$ is $1$-Lipschitz with respect to the adapted Finsler norm using the same reasoning given 
in the proof of Proposition 6.1 in \cite{Avila-Gouezel-13-small-eigenvalues}.
The remainder of the proof of Proposition \ref{prop:AG-6.1} (absent fake zeros) does not rely on the nature of the zeros
with the exception of Lemma 6.2 in \cite{Avila-Gouezel-13-small-eigenvalues}. 
In the remainder of this appendix, we explain why this lemma holds in the case of finitely many additional fake zeros:

\begin{lemma}[compare Lemma 6.2 in \cite{Avila-Gouezel-13-small-eigenvalues}]
\label{lem:AG-6.2}
\ \\
Fix a neighborhood of the identity $U \subset SL_2(\Rbb)$. For every $\delta > 0$, there exists $C_1(\delta)$ such that for all $t>0$, there exists a function $V_{\delta}^{(t)}: \Teich(X, \Sigma, \alpha)\to \Rbb^{+}$ and a scalar $C_2(t, \delta)>0$ satisfying the following property. For all $x \in \Teich^1(X, \Sigma, \alpha)$
\begin{equation*}
\int_0^{2\pi} V_{\delta}^{(t)}(a_t r_{\theta}x) d\theta~
\leq~
C_1e^{-(1-\delta)t}V_{\delta}^{(t)}(x)~ +~ C_2(t,\delta).
\end{equation*}
Moreover, $V_{\delta}^{(t)}(gx) \leq C_1V_{\delta}^{(t)}(x)$ for all 
$x \in \Teich(X, \Sigma, \alpha)$ and $g \in U.$ Finally, there exists a constant $C_{\delta, t}$ such that $V_{\delta}^{(t)}/V \in [C_{\delta,t}^{-1}, C_{\delta,t}],$ where $V = \max \{ 1/ \mathrm{sys}(x), 1\},$ and $\mathrm{sys}(x)$ is the length of the shortest saddle connection. 
\end{lemma} 

Avila and Gou\"ezel attribute this lemma to Eskin-Masur \cite{Eskin-Masur-01-asympt-formulas} and Athreya \cite{Athreya-06-quant-rec}. The only input into the proof that relies on the nature of the zeros is Lemma 7.5 (the system of inequalities) in \cite{Eskin-Masur-01-asympt-formulas}. 

Let $B$ be a domain of a translation surface $(X,\omega)$ with finitely many additional marked points (fake zeros), where a domain is as in \S 6 of~\cite{Eskin-Masur-01-asympt-formulas} but with changes to accommodate finitely many additional fake zeros. Let $a_t = \begin{bmatrix} e^t & 0 \\ 0 & e^{-t} \end{bmatrix}$ and $r_{\theta} = \begin{bmatrix} \cos \theta & \sin \theta \\ -\sin \theta & \cos \theta\end{bmatrix}.$ Define $$V_i(a_t r_{\theta}B) := \max \left( \sup_{K_i} \frac{1}{\lvert \partial K_i \rvert}, 1\right)$$ where the sup is over all admissible subcomplexes as defined in \S 7 of~\cite{Eskin-Masur-01-asympt-formulas} with necessary changes to accommodate finitely many additional fake zeros.\footnote{In Section 7 \cite{Eskin-Masur-01-asympt-formulas}, $V_i$ (written as $\alpha_i$) has a $(1 + \delta)$ exponent on the term with the supremum. For the purpose of verifying Lemma 6.2 in \cite{Avila-Gouezel-13-small-eigenvalues} (with no additional fake zeros), we change this exponent to 1.} 

\begin{lemma}[System of inequalities inclusive of finitely many fake zeros]\label{lem:system-of-inequalities} For any $0 < \delta < 1$ and any $t > 0 $ there exist constants $C_1(\delta)$, $C_2(t, \delta)$, and $C_3(t, \delta)$ such that for any $i$ and any 
$S_0 = a_{t} r_{\phi}B$, 
$$
\frac{1}{2\pi}\int_{0}^{2\pi} V_i(a_{t} r_{\theta}S_0) d\theta~
\leq~
C_1 e^{-(1-\delta)t}V_i(S_0)~ +~ C_2~ +~ C_3 \sum_{j>i} V_j(S_0) \text{.}$$
\end{lemma}

To prove this,
\begin{itemize}
\item the notion of an admissible subcomplex in the definition of the $V_i$'s must hold with finitely many additional zeros. Well-defined admissible subcomplexes follow from validating that \S 6 \cite{Eskin-Masur-01-asympt-formulas} does not depend on the nature of the zeros. 
\item The estimate from Lemma 7.4 \cite{Eskin-Masur-01-asympt-formulas} holds without the $(1+\delta)$ on the integrand.\footnote{We require this change because of how we validate Lemma 6.2 \cite{Avila-Gouezel-13-small-eigenvalues} (with no additional fake zeros) using $V_i$ defined as above.} 
\end{itemize}

Eskin and Masur define their strata using zeros of order $k \geq 1$ and do not initially include marked points.\footnote{Nevertheless, Eskin and Masur apply results from \S 6 \cite{Eskin-Masur-01-asympt-formulas} with an additional fake zero: in the proof of Lemma 9.2 \cite{Eskin-Masur-01-asympt-formulas} they adjoin a point $p$ to the singularity set. They redefine saddle connections accordingly and apply Corollary 6.3 \cite{Eskin-Masur-01-asympt-formulas} to the expanded singularity set.} However, the combinatorial machinery of \S 6 in \cite{Eskin-Masur-01-asympt-formulas} applies with finitely many additional marked points since it imposes no restriction on the order of the zeros. 

The definitions of \S 6.1 and \S 6.2 \cite{Eskin-Masur-01-asympt-formulas} apply to a translation surface with additional fake zeros. As noted in \S B.2, we can triangulate a translation surface using finitely many additional fake zeros in the vertex set.  Lastly, all (oriented) saddle connections $\gamma$ have corresponding holonomy vectors $\rm{hol}(\gamma) = \int_{\gamma} \omega,$ and this allows us to define 
$$
| \gamma_1 \times \gamma_2 | = \left| \rm{Re}(\rm{hol}(\gamma_1)) \cdot \rm{Im}(\rm{hol}(\gamma_2)) - \rm{Im}(\rm{hol}(\gamma_1)) \cdot \rm{Re}(\rm{hol}(\gamma_2)) \right|
$$
for any two saddle connections $\gamma_1$ and $\gamma_2.$ Observe that the definition of $| \gamma_1 \times \gamma_2 |$ is independent of the orientation of the saddle connections. 
%Proposition 6.1 \cite{Eskin-Masur-01-asympt-formulas} follows with this definition of a cross product. Proposition 6.2 \cite{Eskin-Masur-01-asympt-formulas} with finitely many additional fake zeros follows. 
The remaining statements in \S 6 \cite{Eskin-Masur-01-asympt-formulas} follow with this definition of a cross product.  

Lemma 7.4 in \cite{Eskin-Masur-01-asympt-formulas} holds without the $\delta$ in the integrand and with a $\delta$ in the bound on the right-hand side of the inequality. Indeed, 

\begin{prop}
\label{prop:HC} 
\ \\
Fix $0 < \delta < 1.$ Then there exists a constant $C(\delta)$ such that for any $\tau > 0$ and any $v \in \Rbb^2 \setminus \{(0,0)\},$
\begin{equation*}
\frac{1}{2\pi} \int_0^{2\pi} \frac{d\theta}{\| a_{\tau} r_{\theta} v\|} \leq \frac{C(\delta) e^{-\tau(1 - \delta)}}{\| v \|}.
\end{equation*}
\end{prop}

\begin{proof} Let $v = \| v \| \cdot r_{\phi} e_2$ where $e_2 = \begin{bmatrix} 0 \\ 1 \end{bmatrix}.$ Then
\begin{equation*}
\frac{1}{2\pi} \int_{0}^{2\pi} \frac{d\theta}{\| a_{\tau } r_{\theta} v \|} = \frac{1}{2\pi \| v \|} \int_{0}^{2\pi} \frac{d\theta}{\left\| a_{\tau } r_{\theta} r_{\phi} e_2 \right\|} = \frac{1}{2\pi \| v \|} \int_0^{2\pi} \left(e^{-2\tau}\cos^2(\theta) +e^{2\tau}\sin^2(\theta) \right)^{-\frac{1}{2}} d\theta.
\end{equation*}
Using Proposition V.3.1.5 in \cite{HTan92} with $e^{\tau} = a,$ for any $\delta > 0$
\begin{equation*}
\frac{1}{2\pi} \int_0^{2\pi} \left(e^{-2\tau}\cos^2(\theta) +e^{2\tau}\sin^2(\theta) \right)^{-\frac{1}{2}} d\theta \leq C(1+\tau)e^{-\tau} \leq C(\delta) \cdot e^{-\tau(1-\delta)}
\end{equation*}
and the result follows.
\end{proof}

Lemma \ref{lem:system-of-inequalities}, the system of inequalities with finitely many additional fake zeros, follows by applying Proposition \ref{prop:HC} in lieu of Lemma 7.4 \cite{Eskin-Masur-01-asympt-formulas} in the original argument in \cite{Eskin-Masur-01-asympt-formulas}. Note that a second description of how this applies to a translation surface, but without additional marked points, can be found in Dozier's work. See \S 5 in \cite{Doz19}.

Lemma \ref{lem:AG-6.2} and Proposition \ref{prop:AG-6.1} follow with finitely many additional marked points.

\ \medskip

%%%%%%%%%%%%%%%%%%%%%%%%%%%%%%%%%%%%%%%%%%%%%%%%%%%%%%%%%%%%

\end{document}